\documentclass[11pt]{amsart}
\usepackage{geometry}                
\usepackage[parfill]{parskip}    
\usepackage{amssymb, amsthm, amsmath, mathrsfs}
\usepackage{url}
\usepackage{caption}	
\usepackage{subcaption}
\usepackage{color}
\usepackage{comment}
\usepackage[all,cmtip]{xy}
\usepackage{graphicx}
\usepackage{xypic}
\usepackage{tikz}

\usepackage{enumitem}
\setlist[enumerate,1]{label = (\roman*), ref = \roman*}

\newtheorem{theorem}{Theorem}[section]
\newtheorem{corollary}[theorem]{Corollary}
\newtheorem{lemma}[theorem]{Lemma}
\newtheorem{proposition}[theorem]{Proposition}
\newtheorem{remark}[theorem]{Remark}
\newtheorem{property}[theorem]{Property}

\newcommand{\spinc}{\mathrm{Spin}^c}
\newcommand{\mfs}{\mathfrak{s}}
\newcommand{\Z}{\mathbb{Z}}
\newcommand{\T}{\mathcal{T}}

\newcommand{\FF}{\mathbb{F}}

\newcommand{\relspinc}{\underline{\mathrm{Spin}^c}}
\newcommand{\mfv}{\mathfrak{v}}
\newcommand{\mfh}{\mathfrak{h}}
\newcommand{\mft}{\mathfrak{t}}
\newcommand{\X}{\mathbb{X}}

\newcommand{\Xhat}{\widehat{\X}}
\newcommand{\Ahat}{\widehat{A}}
\newcommand{\Bhat}{\widehat{B}}
\newcommand{\vhat}{\widehat{v}}
\newcommand{\hhat}{\widehat{h}}
\newcommand{\Pihat}{\widehat{\Pi}}
\newcommand{\Aplus}{A^+}
\newcommand{\Bplus}{B^+}
\newcommand{\vplus}{v^+}
\newcommand{\hplus}{h^+}
\newcommand{\Xplus}{\X^+}
\newcommand{\XplusN}{\X^{+,N}}
\newcommand{\Piplus}{\Pi^+}
\newcommand{\spin}{\mathrm{Spin}}

\title{Distance one lens space fillings and band surgery on the trefoil knot}

\author{Tye Lidman}
\address{Department of Mathematics\\North Carolina State University\\ Box 8205 \\ Raleigh, NC 27695-8205 \\ USA}
\email{tlid@math.ncsu.edu}

\author{Allison H.\ Moore}
\address{Department of Mathematics \& Applied Mathematics \\ Virginia Commonwealth University \\ 1015 Floyd Avenue, Box 842014  \\ Richmond, VA 23284-2014 \\ USA}
\email{moorea14@vcu.edu}

\author{Mariel Vazquez}
\address{Department of Mathematics \& Department of Microbiology and Molecular Genetics\\University of California, Davis \\ One Shields Avenue  \\ Davis, CA 95616 \\ USA}
\email{mariel@math.ucdavis.edu}


\begin{document}

\begin{abstract} 
We prove that if the lens space $L(n, 1)$ is obtained by a surgery along a knot in the lens space $L(3,1)$ that is distance one from the meridional slope, then $n$ is in $\{-6, \pm 1, \pm 2, 3, 4, 7\}$. 
This result yields a classification of the coherent and non-coherent band surgeries from the trefoil to $T(2, n)$ torus knots and links. 
The main result is proved by studying the behavior of the Heegaard Floer $d$-invariants under integral surgery along knots in $L(3,1)$. 
The classification of band surgeries between the trefoil and torus knots and links is motivated by local reconnection processes in nature, which are modeled as band surgeries. 
Of particular interest is the study of recombination on circular DNA molecules. \end{abstract}

\maketitle


\section{Introduction}
\label{sec:introduction}
The question of whether Dehn surgery along a knot $K$ in the three-sphere yields a three-manifold with finite fundamental group is a topic of long-standing interest, particularly the case of cyclic surgeries. The problem remains open, although substantial progress has been made towards classifying the knots in the three-sphere admitting lens space surgeries \cite{BleilerLitherland, GodaTeragaito, RasmussenLens, OSlens, Baker:SurgeryI, Baker:SurgeryII, Hedden} (see also J. Berge, \emph{unpublished manuscript}, 1990). When the exterior of the knot is Seifert fibered, there may be infinitely many cyclic surgery slopes, such as for a torus knot in the three-sphere \cite{Moser}. In contrast, the celebrated cyclic surgery theorem \cite{CGLS} implies that if a compact, connected, orientable, irreducible three-manifold with torus boundary is not Seifert fibered, then any pair of fillings with cyclic fundamental group has distance at most one. Here, the \emph{distance} between two surgery slopes refers to their minimal geometric intersection number, and a \emph{slope} refers to the isotopy class of an unoriented simple closed curve on the bounding torus. Dehn fillings that are distance one from the fiber slope of a cable space are especially prominent in surgeries yielding prism manifolds \cite{BleilerHodgson:DehnFilling}. Fillings distance one from the meridional slope were also exploited in \cite{baker:poincare} to construct cyclic surgeries on knots in the Poincar\'e homology sphere. 

In this paper, we are particularly interested in Dehn surgeries along knots in $L(3, 1)$ which yield other lens spaces.  The specific interest in $L(3,1)$ is motivated by the study of local reconnection in nature, such as DNA recombination (discussed below).  Note that by taking the knot $K$ to be a core of a genus one Heegaard splitting for $L(3,1)$, one may obtain $L(p,q)$ for all $p,q$.   More generally, since the Seifert structures on $L(3,1)$ are classified \cite{GeigesLange}, one could enumerate the Seifert knots in $L(3,1)$ and use this along with the cyclic surgery theorem to characterize lens space fillings when the surgery slopes are of distance greater than one. 
This strategy does not cover the case where the surgery slopes intersect the meridian of $K$ exactly once. We will refer to these slopes as \emph{distance one surgeries}, also called integral surgeries. 
In this article we are specifically concerned with distance one Dehn surgeries along $K$ in $L(3,1)$ yielding $L(n,1)$. We prove:

\begin{theorem}
\label{thm:main}
The lens space $L(n,1)$ is obtained by a distance one surgery along a knot in the lens space $L(3, 1)$ if and only if $n$ is one of $\pm 1, \pm 2, 3, 4, -6$ or $7$.  
\end{theorem}

While Theorem \ref{thm:main} may be viewed as a generalization of the \emph{lens space realization problem} \cite{Greene}, the result was motivated by the study of reconnection events in nature. Reconnection events are observed in a variety of natural settings at many different scales, for example large-scale magnetic reconnection of solar coronal loops, reconnection of fluid vortices, and microscopic recombination on DNA molecules (e.g. \cite{Li2016, Kleckner2013aa, Shimokawa}). Links of special interest in the physical setting are four-plats, or equivalently two-bridge links, where the branched double covers are lens spaces. In particular, the trefoil $T(2,3)$ is the most probable link formed by any random knotting process \cite{Rybenkov93, Shaw93}, and $T(2,n)$ torus links appear naturally when circular DNA is copied within the cell \cite{AdamsCozz92}. During a reconnection event, two short chain segments, the {\it reconnection sites}, are brought together, cleaved, and the ends are reconnected. When acting on knotted or linked chains, reconnection may change the link type. Reconnection is understood as a band surgery between a pair of links $(L_{1},L_{2})$ in the three-sphere and is modeled locally by a tangle replacement, where the tangle encloses two reconnection sites as illustrated in Figure \ref{fig:resolutions}. Site orientation is important, especially in the physical setting, as explained in Section~\ref{subsec:applications-DNA}. Depending on the relative orientation of the sites, the tangle replacement realizes either a coherent (respectively non-coherent) band surgery, as the links are related by attaching a band (see Figure \ref{fig:resolutions}). More details on the connection to band surgery are included in Section~\ref{sec:applications}.  

\begin{figure}
\begin{center}
\begin{Large}
\begin{tikzpicture}

\node[anchor=south west,inner sep=0] at (0,0) {\includegraphics[width=5.4in]{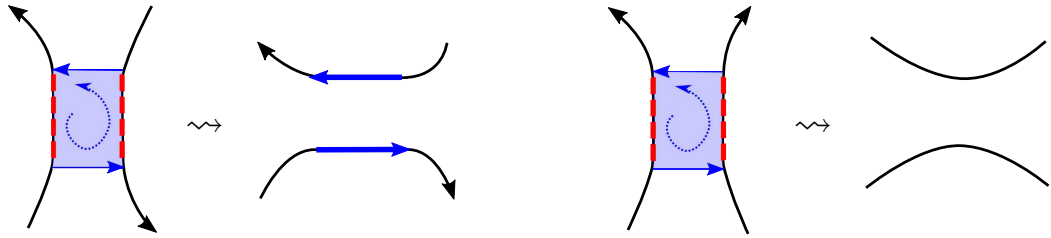}};


\node[label=above right:{$\rightsquigarrow$}] at (2.0,1.0){};
\node[label=above right:{$\rightsquigarrow$}] at (10.0,1.0){};

\end{tikzpicture}
\end{Large}
\end{center}
\caption{The links $L_1$ and $L_2$ differ in a three ball in which a rational tangle replacement is made. Reconnection sites are schematically indicated in red. (Left) A coherent band surgery. (Right) A non-coherent band surgery.}
\label{fig:resolutions}
\end{figure}

We are therefore interested in studying the connection between the trefoil and other torus links by coherent and non-coherent band surgery. 
The Montesinos trick implies that the branched double covers of two links related by a band surgery are obtained by distance one Dehn fillings of a three-manifold with torus boundary.   
Because $L(n, 1)$ is the branched double cover of the torus link $T(2,n)$, Theorem~\ref{thm:main} yields a classification of the coherent and non-coherent band surgeries from the trefoil $T(2,3)$ to $T(2, n)$ for all $n$.  

\begin{corollary}\label{cor}
The torus knot $T(2, n)$ is obtained from $T(2,3)$ by a non-coherent banding if and only if $n$ is $\pm 1$, 3 or 7. The torus link $T(2, n)$ is obtained from $T(2,3)$ by a coherent banding if and only if $n$ is $\pm 2$, 4 or -6.
\end{corollary}
\begin{proof}Theorem \ref{thm:main} obstructs the existence of any coherent or non-coherent banding from $T(2, 3)$ to $T(2,n)$ when $n$ is not one of the integers listed in the statement.  Bandings illustrating the remaining cases are shown in Figures \ref{fig:bandings-1} and \ref{fig:bandings-2}.
\end{proof}

In our convention $T(2, 3)$ denotes the right-handed trefoil. The statement for the left-handed trefoil is analogous after mirroring.  Note that Corollary~\ref{cor} certifies that each of the lens spaces listed in Theorem~\ref{thm:main} is indeed obtained by a distance one surgery from $L(3,1)$.  We remark that \emph{a priori}, a knot in $L(3,1)$ admitting a distance one lens space surgery to $L(n, 1)$ does not necessarily descend to a band move on $T(2,3)$ under the covering involution. 

When $n$ is even, if the linking number of $T(2,n)$ is $+n/2$, Corollary~\ref{cor} follows as a consequence of the behavior of the signature of a link \cite{Murasugi}. If the linking number is instead $-n/2$, Corollary \ref{cor} follows from the characterization of coherent band surgeries between $T(2,n)$ torus links and certain two-bridge knots in \cite[Theorem 3.1]{DIMS}.  While both coherent and non-coherent band surgeries have biological relevance, more attention in the literature has been paid to the coherent band surgery model (see for example \cite{IshiharaShimokawa, DIMS, Shimokawa, ISV, BuckIshihara2015, BIRS, Stolz2017}). This is due in part to the relative difficulty in working with non-orientable surfaces, as is the case with non-coherent band surgery on knots.

\begin{figure}
\begin{center}
\begin{tikzpicture}

\node[anchor=south west,inner sep=0] at (0,0) {\includegraphics[height=1.4in]{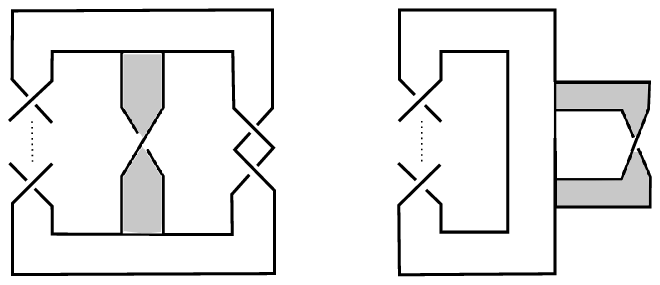} \qquad \includegraphics[height=1.4in]{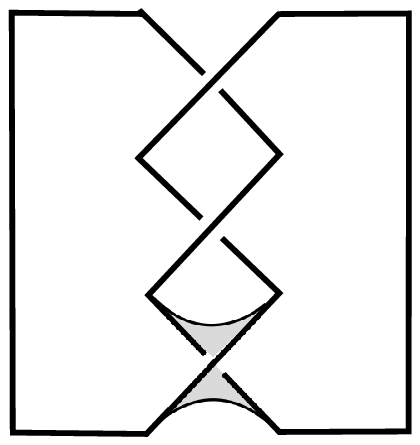} };


\node[label=above right:{$n$}] at (-0.4,1.4){};
\node[label=above right:{$n$}] at (4.6,1.4){};

\end{tikzpicture}
\end{center}

\caption{Non-coherent bandings: (Left) $T(2,n-2)$ to $T(2,n+2)$. (Center) $T(2,n)$ to itself. (Right) $T(2, 3)$ to the unknot.}
\label{fig:bandings-1}
\end{figure}

\begin{figure}
\includegraphics[height = 1.4in]{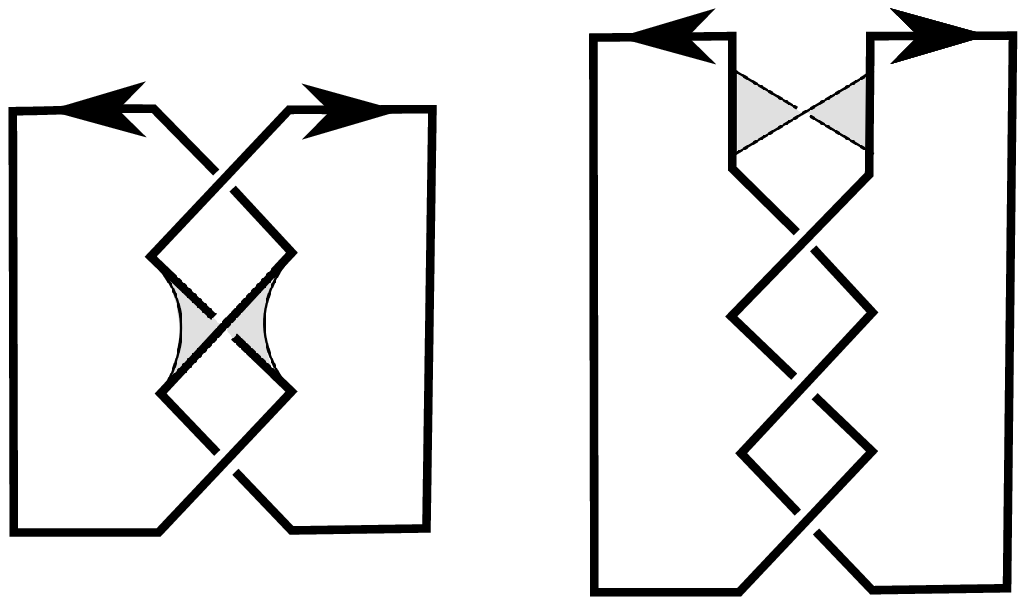}\qquad\qquad \includegraphics[height=1.4in]{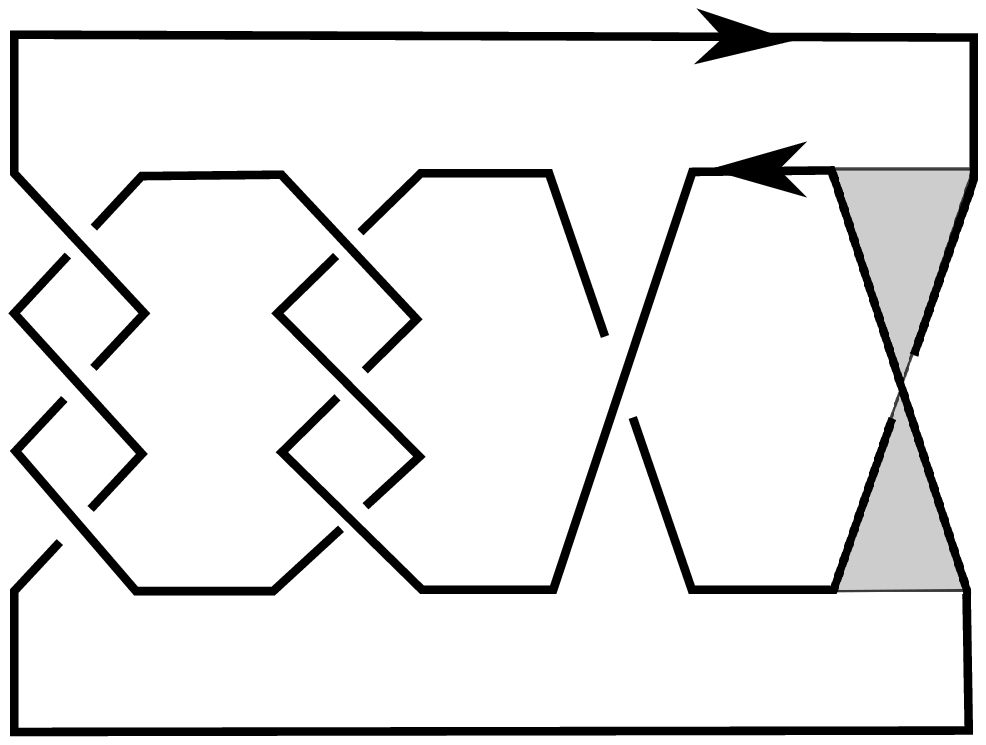}
\caption{Coherent bandings: (Left)  $T(2,3)$ to $T(2, 2)$. (Center) $T(2,3)$ to $T(2,4)$. (Right) $T(2, -6)$ to $T(2,3)$ (see also \cite[Theorem 5.10]{DIMS}).}
\label{fig:bandings-2}
\end{figure}

\textbf{Overview of main result.} The key ingredients in the proof of Theorem \ref{thm:main} are a set of formulas, namely \cite[Proposition 1.6]{NiWu} and its generalizations in Propositions \ref{prop:3k-1-surgery-formula} and \ref{prop:3k+1-surgery-formula}, which describe the behavior of $d$-invariants under certain Dehn surgeries. Recall that a \emph{$d$-invariant} or \emph{correction term} is an invariant of the pair $(Y, \mathfrak{t})$, where $Y$ is an oriented rational homology sphere and $\mathfrak{t}$ is an element of $\spinc(Y)\cong H^2(Y; \Z)$.  More generally, each $d$-invariant is a $\spinc$ rational homology cobordism invariant. This invariant takes the form of a rational number given by the minimal grading of an element in a distinguished submodule of the Heegaard Floer homology, $HF^{+}(Y, \mathfrak{t})$ \cite{OSAbs}.  Work of Ni-Wu \cite{NiWu} relates the $d$-invariants of surgeries along a knot $K$ in $S^3$, or more generally a null-homologous knot in an L-space, with a sequence of non-negative integer-valued invariants $V_i$, due to Rasmussen (see for reference the \emph{local h-invariants} in \cite{RasmussenThesis} or \cite{NiWu}).  

With this we now outline the proof of Theorem \ref{thm:main}. Suppose that $L(n, 1)$ is obtained by surgery along a knot $K$ in $L(3, 1)$. As is explained in Lemma \ref{lem:homology-1}, the class of $|n|$ modulo 3 determines whether or not $K$ is homologically essential. When $n\equiv 0 $ (mod 3), we have that $K$ is null-homologous. In this case, we take advantage of the Dehn surgery formula due to Ni-Wu mentioned above and a result of Rasmussen \cite[Proposition 7.6]{RasmussenThesis} which bounds the difference in the integers $V_i$ as $i$ varies. Then by comparing this to a direct computation of the correction terms for the lens spaces of current interest, we obstruct a surgery from $L(3,1)$ to $L(n,1)$ for $n \neq 3$ or $-6$. 

When $|n|\equiv \pm1 \pmod{3}$, we must generalize the correction term surgery formula of Ni-Wu to a setting where $K$ is homologically essential. The technical work related to this generalization makes use of the mapping cone formula for rationally null-homologous knots \cite{OSRational}, and is contained in Section \ref{sec:mappingcone}. This surgery formula is summarized in Propositions~\ref{prop:3k-1-surgery-formula} and \ref{prop:3k+1-surgery-formula}, which we then use in a similar manner as in the null-homologous case. We find that among the oriented lens spaces of order $\pm 1$ modulo 3, $\pm L(2, 1)$, $L(4,1)$ and $L(7,1)$ are the only nontrivial lens spaces with a distance one surgery from $L(3,1)$, completing the proof of Theorem~\ref{thm:main}.


\subsection*{Outline.} In Section \ref{sec:homology}, we establish some preliminary homological information that will be used throughout and study the $\spinc$ structures on the two-handle cobordisms arising from distance one surgeries. Section \ref{sec:proof} contains the proof of Theorem \ref{thm:main}, separated into the three cases as described above. Section \ref{sec:mappingcone} contains the technical arguments pertaining to Propositions~\ref{prop:3k-1-surgery-formula} and \ref{prop:3k+1-surgery-formula}, which compute $d$-invariants of certain surgeries along a homologically essential knot in $L(3,1)$.  Lastly, in Section \ref{sec:applications} we present the biological motivation for the problem in relation with DNA topology and discuss coherent and non-coherent band surgeries more precisely.


\section{Preliminaries}
\label{sec:homology}

\subsection{Homological preliminaries}
\label{subsec:homology}
We begin with some basic homological preliminaries on surgery on knots in $L(3,1)$.  This will give some immediate obstructions to obtaining certain lens spaces by distance one surgeries.  Here we will also set some notation.  All singular homology groups will be taken with $\mathbb{Z}$-coefficients except when specified otherwise.  

Let $Y$ denote a rational homology sphere. First, we will use the torsion linking form on homology:  
\begin{eqnarray*}
	\ell k : H_1(Y) \times H_1(Y) &\to& \mathbb{Q}/\Z. 
\end{eqnarray*}
See \cite{Friedl} for a thorough exposition on this invariant.

In the case that $H_1(Y)$ is a cyclic group, it is enough to specify the linking form by determining the value $\ell k(x, x)$ for a generator $x$ of $H_1(Y)$ and extending by bilinearity. 
Consequently, if two rational homology spheres $Y_1$ and $Y_2$ have cyclic first homology with linking forms given by $\frac{n}{p}$ and $\frac{m}{p}$, where $p > 0$, then the two forms are equivalent if and only if $n \equiv ma^2 \pmod{p}$ for some integer $a$ with $\gcd(a,p) = 1$.  We take the convention that $L(p,q)$ is obtained by $p/q$-surgery on the unknot, and that the linking form is given by $q/p$.\footnote{We choose this convention to minimize confusion with signs.  The deviation from $-q/p$ to $q/p$ is irrelevant for our purposes, since this change will uniformly switch the sign of each linking form computed in this section.  Because $\ell k_1$ and $\ell k_2$ are equivalent if and only if $-\ell k_1$ and $-\ell k_2$ are equivalent, this will not affect the results.} Following these conventions, $p/q$-surgery on any knot in an arbitrary integer homology sphere has linking form $q/p$ as well.

Let $K$ be any knot in $Y = L(3,1)$. The first homology class of $K$ is either trivial or it generates $H_1(Y) = \Z/3$, in which case we say that $K$ is \emph{homologically essential}. When $K$ is null-homologous, then the surgered manifold $Y_{p/q}(K)$ is well-defined and $H_1(Y_{p/q}(K)) = \Z/3 \oplus \Z/p$. When $K$ is homologically essential, there is a unique such homology class up to a choice of an orientation on $K$. The exterior of $K$ is denoted $M = Y-\mathcal{N}(K)$ and because $K$ is homologically essential, $H_1(M) = \Z$. Recall that the rational longitude $\ell$ is the unique slope on $\partial M$ which is torsion in $H_1(M)$. In our case, the rational longitude $\ell$ is null-homologous in $M$. We write $m$ for a choice of dual peripheral curve to $\ell$ and take $(m, \ell)$ as a basis for $H_1(\partial M)$. Let $M(pm + q \ell)$ denote the Dehn filling of $M$ along the curve $pm + q\ell$, where $gcd(p,q)=1$. It follows that $H_1(M(pm + q \ell)) = \Z/p$ and that the linking form of $M(pm+q \ell)$ is equivalent to $q/p$ when $p \neq 0$.  Indeed, $M(pm + q \ell)$ is obtained by $p/q$-surgery on a knot in an integer homology sphere, namely the core of the Dehn filling $M(m)$.    

Recall that we are interested in the distance one surgeries to lens spaces of the form $L(n,1)$.  Therefore, we first study when distance one surgery results in a three-manifold with cyclic first homology.  We begin with an elementary homological lemma.

\begin{lemma}\label{lem:homology-1}
Fix a non-zero integer $n$.  Suppose that $Y'$ is obtained from $Y = L(3,1)$ by a distance one surgery on a knot $K$ and that $H_1(Y') = \Z/n$.  
\begin{enumerate}
\item\label{homology-1-1}
 If $n = 3k \pm 1$, then $K$ is homologically essential.

\item\label{homology-1-2} If $K$ is homologically essential, the slope of the meridian on $M$ is $3m + (3r+1)\ell$ for some integer $r$.  Furthermore, there is a choice of $m$ such that $r = 0$.  

\item\label{homology-1-3} With the meridian on $M$ given by $3m+ \ell$ as above, then if $n = 3k+1$ (respectively $n = 3k-1$), the slope inducing $Y'$ on $M$ is $(3k+1)m + k \ell$ (respectively $(3k-1)m + k \ell$). 

\item\label{homology-1-4}
 If $n = 3k$, then $K$ is null-homologous and the surgery coefficient is $\pm k$.  Furthermore, $\gcd(k,3) = 1$.   
\end{enumerate}  
\end{lemma} 
\begin{proof}

\eqref{homology-1-1}
This follows since surgery on a null-homologous knot in $Y$ has $H_1(Y_{p/q}(K)) = \Z/3 \oplus \Z/p$.

\eqref{homology-1-2}
By the discussion preceding the lemma, we have that the desired slope must be $3m + q \ell$ for some $q$ relatively prime to 3.  In this case, $M(3m + q \ell)$ has linking form equivalent to $1/3$ or $2/3$, depending on whether $q \equiv 1$ or $2 \pmod{3}$.  Since 2 is not a square mod 3, we see that the linking form $2/3$ is not equivalent to that of $1/3$, which is the linking form of $L(3,1)$.  Therefore, $q \equiv 1 \pmod{3}$ and the meridian is $3m + (3r+1)\ell$ for some $r$.  By instead using the peripheral curve $m' = m + r\ell$, which is still dual to $\ell$, we see that the meridian is given by $3m' + \ell$.

\eqref{homology-1-3} By the previous item, we may choose $m$ such that the meridional slope of $K$ on $M$ is given by $3m + \ell$.  Now write the slope on $M$ yielding $Y'$ as $(3k\pm 1)m + q \ell$.  In order for this slope to be distance one from $3m + \ell$, we must have that $q = k$.    

\eqref{homology-1-4} Note that {\em if} $K$ is null-homologous, then the other two conclusions easily hold since $H_1(Y') = \Z/3 \oplus \Z/k$.  Therefore, we must show that $K$ cannot be homologically essential.  If $K$ was essential, then the slope on the exterior would be of the form $3km + s \ell$ for some integer $s$.  The distance from the meridian is then divisible by 3, which is a contradiction.
\end{proof}

In this next lemma, we use the linking form to obtain a surgery obstruction.  

\begin{lemma}
\label{lem:homology-2}
Fix a non-zero odd integer $n$.  Let $K$ be a knot in $Y = L(3,1)$ with a distance one surgery to $Y'$ having $H_1(Y') = \Z/n$ and linking form equivalent to $\frac{sgn(n)}{|n|}$.    
If $n \equiv 1 \pmod{3}$, then $n > 0$.      
\end{lemma}
\begin{proof}
Suppose that $n < 0$. Write $n=1-3j$ with $j>0$. By assumption, the linking form of $Y'$ is $-1/(3j-1)$.  By Lemma~\ref{lem:homology-1}\eqref{homology-1-3}, the linking form of $Y'$ is also given by $j/(3j-1)$.  Consequently, $-j$ is a square modulo $3j-1$ or equivalently, $-3$ is a square modulo $3j-1$, as $-3$ is the inverse of $-j$. 
Because $n$ is odd, the law of quadratic reciprocity implies that for any prime $p$ dividing $3j-1$, we have that $p \equiv 1 \pmod{3}$.  This contradicts the fact that $3j-1 \equiv -1 \pmod{3}$.
\end{proof}

\begin{remark} By an argument analogous to Lemma \ref{lem:homology-2}, one can prove that if $n = 3k-1$ is odd, then $n \equiv 1$ or $11 \pmod{12}$.
\end{remark}

Lemma~\ref{lem:homology-2} does not hold if $n$ is even.  This can be seen since $-L(2,1) \cong L(2,1)$ is obtained from a distance one surgery on a core of the genus one Heegaard splitting of $L(3,1)$.  In Section~\ref{subsec:homology-d} we will be able to obtain a similar obstruction in the case that $n$ is even.

\subsection{The four-dimensional perspective}
Given a distance one surgery between two three-manifolds, we let $W$ denote the associated two-handle cobordism. For details on the framed surgery diagrams and associated four-manifold invariants used below, see \cite{GS}. 

\begin{lemma}\label{lem:cobordism}
Suppose that $Y'$ is obtained from a distance one surgery on $L(3,1)$.   
\begin{enumerate}
\item\label{cob:definite} If $|H_1(Y')| = 3k-1$, then $W$ is positive-definite, whereas if $|H_1(Y')| = 3k+1$, then $W$ is negative-definite. 
\item\label{cob:spin} The order of $H_1(Y')$ is even if and only if $W$ is $\spin$.  
\end{enumerate}
\end{lemma}
\begin{proof}
\eqref{cob:definite} In either case, Lemma~\ref{lem:homology-1} implies that $Y'$ is obtained by integral surgery on a homologically essential knot $K$ in $L(3,1)$.  First, $L(3,1)$ is the boundary of a four-manifold $N$, which is a $+3$-framed two-handle attached to $B^4$ along an unknot.  Let $Z$ denote $N \cup W$.  Since $b_2^\pm(Z) = b_2^\pm(N) + b_2^\pm(W)$, we see that $W$ is positive-definite (respectively negative-definite) if and only if $b_2^+(Z)$ is equal to 2 (respectively 1).  

Since $K$ is homologically essential, after possibly reversing the orientation of $K$ and handlesliding $K$ over the unknot, we may present $Y'$ by surgery on a two-component link with linking matrix
\[
Q = \begin{pmatrix}
3 & 1 \\ 1 & c
\end{pmatrix},   
\]
which implies that the order of $H_1(Y')$ is $|3c-1|$.  Since the intersection form of $Z$ is presented by $Q$, we see that $b_2^+(Z)$ equals 2 (respectively 1) if and only if $c > 0$ (respectively $c \leq  0$).  The claim now follows.  

\eqref{cob:spin} We will use the fact that an oriented four-manifold whose first homology has no 2-torsion is $\spin$ if and only if its intersection form is even.  First, note that $H_1(W)$ is a quotient of $\Z/3$, so $H_1(W;\Z/2) = 0$.  Next, view $L(3,1)$ as the boundary of the $\spin$ four-manifold $X$ obtained from attaching $-2$-framed two-handles to $B^4$ along the Hopf link.  This is indeed $\spin$, because $X$ is simply-connected and has even intersection form.  After attaching $W$ to $X$, we obtain a presentation for the intersection form of $W \cup X$:
\[
Q_{W \cup X} = \begin{pmatrix}
-2 & 1 & a \\ 1 & - 2 & b \\ a & b & c
\end{pmatrix}.
\]
Since this matrix presents $H_1(Y')$, we compute that $|H_1(Y')|$ is even if and only if $c$ is even if and only if the intersection form of $W \cup X$ is even. Since $X$ is $\spin$ and we are attaching $W$ along a $\Z/2$-homology sphere, we see that the simply-connected four-manifold $W \cup X$ is $\spin$ if and only if $W$ is $\spin$.  Consequently, $|H_1(Y')|$ is even if and only if $W$ is $\spin$.  \end{proof}

\subsection{$d$-invariants, lens spaces, and $\spin$ manifolds}\label{subsec:homology-d}
As mentioned in the introduction, the main invariant that we will use is the $d$-invariant, $d(Y,\mft)$, of a $\spinc$ rational homology sphere $(Y,\mft)$. These invariants are intrinsically related with the intersection form of any smooth, definite four-manifold bounding $Y$ \cite{OSAbs}. 
In some sense, the $d$-invariants can be seen as a refinement of the torsion linking form on homology. For homology lens spaces, this notion can be made more precise as in \cite[Lemma 2.2]{LidmanSivek}. 

We assume familiarity with the Heegaard Floer package and the $d$-invariants of rational homology spheres, referring the reader to \cite{OSAbs} for details.  We will heavily rely on the following recursive formula for the $d$-invariants of a lens space. 
\begin{theorem}[Ozsv\'ath-Szab\'o, Proposition 4.8 in \cite{OSAbs}]\label{thm:d-lens}
Let $p>q > 0$ be relatively prime integers.  Then, there exists an identification $\spinc(L(p,q)) \cong \Z/p$ such that 
\begin{equation}\label{eq:d-lens}
d(L(p,q),i) = -\frac{1}{4} + \frac{(2i+1-p-q)^2}{4pq} - d(L(q,r),j)
\end{equation}
for $0 \leq i < p+q$.  Here, $r$ and $j$ are the reductions of $p$ and $i \pmod{q}$ respectively.  
\end{theorem}
Under the identification in Theorem~\ref{thm:d-lens}, it is well-known that the self-conjugate $\spinc$ structures on $L(p,q)$ correspond to the integers among
\begin{equation}\label{eq:spin-formula}
\frac{p + q - 1}{2} \text{ and } \frac{q - 1}{2}.
\end{equation}
(See for instance \cite[Equation (3)]{DoigWehrli}.)

For reference, following \eqref{eq:d-lens}, we give the values of $d(L(n,1),i)$, including $d(L(n,1),0)$, for $n > 0$:
\begin{align}\label{eq:d-L(n,1)}
d(L(n,1),i) &= -\frac{1}{4} + \frac{(2i-n)^2}{4n} \\
\nonumber d(L(n,1),0) &= \frac{n-1}{4}. 
\end{align}
It is useful to point out that $d$-invariants change sign under orientation-reversal \cite{OSAbs}. 

Using the work of this section, we are now able to heavily constrain distance one surgeries from $L(3,1)$ to $L(n,1)$ in the case that $n$ is even.  
\begin{proposition}\label{prop:even}
Suppose that there is a distance one surgery between $L(3,1)$ and $L(n,1)$ where $n$ is an even integer.  Unless $n = 2$ or $4$, we have $n < 0$.  In the case that $n < 0$, the two-handle cobordism from $L(3,1)$ to $L(n,1)$ is positive-definite and the unique $\spin$ structure on $L(n,1)$ which extends over this cobordism corresponds to $i = |\frac{n}{2}|$.    
\end{proposition}

A technical result that we need is established first, which makes use of Lin's $\mathrm{Pin(2)}$-equivariant monopole Floer homology \cite{Lin}.  

\begin{lemma}\label{lem:even-Lspace}
Let $(W,\mfs): (Y,\mft) \to (Y',\mft')$ be a $\spin$ cobordism between L-spaces satisfying $b_2^+(W) = 1$ and $b_2^-(W) = 0$.  Then 
\begin{equation}\label{eq:spin-1/4}
d(Y',\mft') - d(Y,\mft) = -\frac{1}{4}.  
\end{equation}
\end{lemma}
\begin{proof}
By \cite[Theorem 5]{LinSurgery}, we have that 
\[
\alpha(Y',\mft') - \beta(Y,\mft) \geq -\frac{1}{8},   
\]  
where $\alpha$ and $\beta$ are Lin's adaptation of the Manolescu invariants for $\mathrm{Pin(2)}$-equivariant monopole Floer homology.  Conveniently, for L-spaces, $\alpha = \beta = \frac{d}{2}$ \cite{Cristofaro, Ramos, Lin, HuangRamos}.  Thus, we have 
\begin{equation}\label{eq:spin-1/4-inequality}
d(Y',\mft') - d(Y,\mft) \geq -\frac{1}{4}.
\end{equation}

On the other hand, we may reverse orientation on $W$ to obtain a negative-definite $\spin$ cobordism $(-W,\mfs) : (-Y,\mft) \to (-Y',\mft')$.  Therefore, we have from \cite[Theorem 9.6]{OSAbs} that 
\[
	d(-Y',\mft') - d(-Y,\mft) \geq \frac{c_1(\mfs)^2 + b_2(-W)}{4} = \frac{1}{4}.  
\]
Combined with \eqref{eq:spin-1/4-inequality}, this completes the proof.  
\end{proof}

\begin{proof}[Proof of Proposition~\ref{prop:even}]

For completeness, we begin by dispensing with the case of $n = 0$, i.e., $S^2 \times S^1$.  This is obstructed by Lemma~\ref{lem:homology-1}, since no surgery on a null-homologous knot in $L(3,1)$ has torsion-free homology.  

Therefore, assume that $n \neq 0$.  The two-handle cobordism $W$ is $\spin$ by Lemma~\ref{lem:cobordism}.  First, suppose that $b_2^+(W) = 1$ (and consequently $b_2^-(W) = 0$), so that we may apply Lemma \ref{lem:even-Lspace}.  Because $\mfs$ on $W$ restricts to self-conjugate $\spinc$ structures $\mft$ and $\mft'$ on $Y$ and $Y'$, \eqref{eq:spin-formula} and \eqref{eq:spin-1/4} imply that
\begin{equation}\label{eq:spin-n,i}
d(L(n,1),i) - d(L(3,1),0) = -\frac{1}{4},
\end{equation}
where $i$ must be one of $0$ or $|\frac{n}{2}|$. Applying Equation \eqref{eq:d-L(n,1)} to $L(3,1)$, we conclude that $d(L(n,1),i) = \frac{1}{4}$. 

If $i = 0$, Equation~\eqref{eq:d-L(n,1)} applied to $L(n,1)$ implies that $d(L(n,1),0) = \frac{|n|-1}{4}$ for $n>0$ and $\frac{1-|n|}{4}$ for $n<0$. The only solution agreeing with \eqref{eq:spin-n,i} is when $n=+2$. 
If $i = \frac{|n|}{2}$, Equation~\eqref{eq:d-L(n,1)} implies that $d(L(n,1),i)$ is $-\frac{1}{4}$ for $n>0$ and $\frac{1}{4}$ for $n<0$, and so \eqref{eq:spin-n,i} holds whenever $n < 0$.  Note that in this case, $W$ is positive-definite.

%

Now, suppose that $b_2^+(W) = 0$.  Therefore, we apply Lemma~\ref{lem:even-Lspace} instead to $-W$ to see that 
\[
-d(L(n,1),i) + d(L(3,1),0) = -\frac{1}{4},
\]
where again, $i = 0$ or $|\frac{n}{2}|$.  In this case, there is a unique solution given by $n = +4$ when $i = 0$.  This completes the proof.
\end{proof}

\subsection{$d$-invariants and surgery on null-homologous knots}
Throughout the rest of the section, we assume that $K$ is a null-homologous knot in a rational homology sphere $Y$.  By Lemma~\ref{lem:homology-1}, this will be relevant when we study surgeries to $L(n,1)$ with $n \equiv 0 \pmod{3}$.  Recall that associated to $K$, there exist non-negative integers $V_{\mft,i}$ for each $i \in \Z$ and $\mft \in \spinc(Y)$ satisfying the following property:
\begin{property}[Proposition 7.6 in \cite{RasmussenThesis}]
\label{thm:rasmussen-localh}
\[V_{\mft,i} \geq V_{\mft,i+1} \geq V_{\mft,i} -1. \]
\end{property} 

When $K$ is null-homologous in $Y$, the set of $\spinc$ structures $\spinc(Y_p(K))$ is in one-to-one correspondence with $\spinc(Y) \oplus \Z/p$. The projection to the first factor comes from considering the unique $\spinc$ structure on $Y$ which extends over the two-handle cobordism $W_p(K): Y \rightarrow Y_p(K)$ to agree with the chosen $\spinc$ structure on $Y_p(K)$.  With this in mind, we may compute the $d$-invariants of $Y_p(K)$ as follows.  The result below was proved for knots in $S^3$, but the argument immediately generalizes to the situation considered here. 

\begin{proposition}[Proposition 1.6 in \cite{NiWu}]
\label{thm:ni-wu}
Fix an integer $p > 0$ and a self-conjugate $\spinc$ structure $\mft$ on an L-space $Y$. Let $K$ be a null-homologus knot in $Y$.   Then, there exists a bijective correspondence $i \leftrightarrow \mft_i$ between $\Z/p\Z$ and the $\spinc$ structures on $\spinc(Y_p(K))$ that extend $\mft$ over $W_p(K)$ such that
	\begin{equation}
	\label{ni-wu-formula}
	d(Y_p(K),\mft_i) = d(Y,\mft) + d(L(p,1),i) - 2N_{\mft, i} 
	 \end{equation} 
where $N_{\mft,i} = \max\{V_{\mft,i}, V_{\mft,p-i}\}$.  Here, we assume that $0 \leq i < p$.  
\end{proposition}

In order to apply Proposition~\ref{thm:ni-wu}, we must understand the identifications of the $\spinc$ structures precisely.  In particular, the correspondence between $i$ and $\mft_i$ is given in \cite[Theorem 4.2]{OSInteger}.  Let $\mfs$ be a $\spinc$ structure on $W_p(K)$ which extends $\mft$ and let $\mft_i$ be the restriction to $Y_p(K)$.  Then, we have from \cite[Theorem 4.2]{OSInteger} that $i$ is determined by  
\begin{equation}\label{eq:3k-cobordism-spinc}
\langle c_1(\mfs), [\widehat{F}] \rangle + p \equiv 2i \pmod{2p},
\end{equation}
where $[\widehat{F}]$ is the surface in $W_p(K)$ coming from capping off a Seifert surface for $K$.  For this to be well-defined, we must initially choose an orientation on $K$, but the choice will not affect the end result.    

Before stating the next lemma, we note that if $Y$ is a $\Z/2$-homology sphere, then $H^1(W_p(K);\Z/2) = 0$, and thus there is at most one $\spin$ structure on $W_p(K)$.  If $p$ is even, $W_p(K)$ is $\spin$, since the intersection form is even and $H_1(W_p(K);\Z/2) = 0$.  Further, $Y_p(K)$ admits exactly two $\spin$ structures, and thus exactly one extends over $W_p(K)$.  
 
\begin{lemma}\label{lem:spin-3k}
Let $K$ be a null-homologous knot in a $\Z/2$-homology sphere $Y$.  Let $\mft$ be the self-conjugate $\spinc$ structure on $Y$, and let $\mft_0$ be the $\spinc$ structure on $Y_p(K)$ described in Proposition~\ref{thm:ni-wu} above.  
\begin{enumerate}
\item\label{spin-3k:1} Then, $\mft_0$ is self-conjugate on $Y_p(K)$.  
\item\label{spin-3k:2} The $\spin$ structure $\mft_0$ does {\em not} extend to a $\spin$ structure over $W_p(K)$.
\end{enumerate}
\end{lemma}
\begin{proof}
\eqref{spin-3k:1} By \eqref{eq:3k-cobordism-spinc}, we see that if $\mfs$ extends $\mft_0$ over $W_p(K)$, 
\[
\langle c_1(\mfs), [\widehat{F}] \rangle \equiv  -p \pmod{2p}.
\]
Note that $\overline{\mfs}$ extends $\overline{\mft_0}$ over $W_p(K)$ and restricts to $\mft$ on $Y$, since $\mft$ is self-conjugate.  The above equation now implies that  
\[
\langle c_1(\overline{\mfs}), [\widehat{F}] \rangle \equiv  p \equiv -p  \pmod{2p}.
\]
In the context of \eqref{eq:3k-cobordism-spinc}, $i = 0$.  Consequently, we must have that $\overline{\mfs}$ also restricts to $\mft_0$ on $Y_p(K)$.  Of course, this implies that $\mft_0$ is self-conjugate.  

\eqref{spin-3k:2} By \eqref{eq:3k-cobordism-spinc}, we deduce that for a $\spin$ structure that extends $\mft_i$ over $W_p(K)$, $p \equiv 2i \pmod{2p}$.  Since we consider $0 \leq i \leq p$, we have that $i = \frac{p}{2} \neq 0$.  Consequently, $\mft_0$ cannot extend to a $\spin$ structure on $W_p(K)$.   
\end{proof}

\section{The proof of Theorem~\ref{thm:main}}
We now prove Theorem~\ref{thm:main} through a case analysis depending on the order of the purported lens space surgery modulo 3.  
\label{sec:proof}

\subsection{From $L(3,1)$ to $L(n, 1)$ where $|n|\equiv 0$ (mod 3)}
The goal of this section is to prove:
\begin{proposition}\label{prop:3k}
There is no distance one surgery from $L(3,1)$ to $L(n,1)$, where $|n| =3k$, except when $n = 3$ or $-6$.  
\end{proposition}
\begin{proof}
Let $K$ be a knot in $L(3,1)$ with a distance one surgery to $L(n,1)$ where $|n| = 3k$.  By Lemma \ref{lem:homology-1}\eqref{homology-1-4}, we know that $K$ is null-homologous and the surgery coefficient is $\pm k/1$, and by Proposition~\ref{prop:even}, $k \neq 0$.  

The proof now follows from the four cases addressed in Propositions \ref{prop:3k++}, \ref{prop:3k--}, \ref{prop:3k+-} and \ref{prop:3k-+} below, which depend on the sign of $n$ and the sign of the surgery on $L(3,1)$. We obtain a contradiction in each case, except when $n = 3$ or $-6$.  These exceptional cases can be realized through the band surgeries in Figures~\ref{fig:bandings-1} and \ref{fig:bandings-2} respectively.   
\end{proof}

We now proceed through the case analysis described in the proof of Proposition~\ref{prop:3k}.  

\begin{proposition}\label{prop:3k++}
If $k \geq 2$, then $L(3k,1)$ cannot be obtained by $+k/1$-surgery on a null-homologous knot in $L(3,1)$.   
\end{proposition}
\begin{proof}
By Proposition~\ref{prop:even}, $3k$ cannot be even, so we may assume that $L(3k,1)$ is obtained by $k$-surgery on a null-homologous knot $K$ in $Y = L(3,1)$ for $k$ odd.  Consequently, there are unique self-conjugate $\spinc$ structures on $L(3k,1), L(3,1),$ and $L(k,1)$.  By \eqref{eq:spin-formula}, Proposition~\ref{thm:ni-wu}, and Lemma~\ref{lem:spin-3k}, 
\begin{equation}\label{eq:three-ds}
d(L(3k,1),0) \leq d(L(3,1),0) + d(L(k,1),0).
\end{equation}
Using the $d$-invariant formula \eqref{eq:d-L(n,1)}, when $k \geq 2$, we have 
\begin{equation*}
d(L(3k,1),0) - d(L(3,1),0) - d(L(k,1),0) = \frac{-1+3k}{4} - \frac{-1+3}{4} - \frac{-1+k}{4} > 0,
\end{equation*}
which contradicts \eqref{eq:three-ds}.   
\end{proof} 

\begin{proposition}\label{prop:3k--}
If $k \geq 1$, then $L(-3k,1)$ cannot be obtained by $-k/1$-surgery on a null-homologous knot in $L(3,1)$.   
\end{proposition}
\begin{proof}
Suppose that $L(-3k,1)$ is obtained by $-k/1$-surgery on a null-homologous knot in $L(3,1)$.    By Proposition~\ref{prop:even}, we cannot have that $3k$ is even.  Indeed, in the current case, the associated two-handle cobordism is negative-definite.  Therefore, $3k$ is odd, and we have unique self-conjugate $\spinc$ structures on $L(3k,1)$ and $L(k,1)$.  

By reversing orientation, $L(3k,1)$ is obtained by $+k$-surgery on a null-homologous knot in $L(-3,1)$.  We may now repeat the arguments of Proposition~\ref{prop:3k++} with a slight change.  We obtain that  
\[
d(L(3k,1),0) \leq -d(L(3,1),0) + d(L(k,1),0).
\]
By direct computation, 
\[ 
	d(L(3k,1),0) + d(L(3,1),0) - d(L(k,1),0) = \frac{-1+3k}{4} + \frac{1}{2} - \frac{-1+k}{4}  > 0.
\]
Again, we obtain a contradiction.  
\end{proof}

\begin{proposition}\label{prop:3k+-}
If $k \geq 2$, then $L(3k,1)$ cannot be obtained by $-k/1$-surgery on a null-homologous knot in $L(3,1)$.
\end{proposition}
\begin{proof}
As in the previous two propositions, Proposition~\ref{prop:even} implies that $k$ cannot be even.  Therefore, we assume that $k$ is odd.  We will equivalently show that if $k \geq 3$ is odd, then $L(-3k,1)$ cannot be obtained by $+k/1$-surgery on a null-homologous knot in $L(-3,1)$.   

Again, consider the statement of Proposition \ref{thm:ni-wu} in the case of the unique self-conjugate $\spinc$ structure on $L(3k,1)$.  Writing $\mft$ for the self-conjugate $\spinc$ structure on $L(-3,1)$, Equations \eqref{eq:d-L(n,1)} and  \eqref{ni-wu-formula} yield
\begin{equation*}
	2N_{\mft, 0} = d(L(3k, 1), 0) - d(L(3,1), 0) + d(L(k, 1),0) = \left(-\frac{1}{4} + \frac{3k}{4}\right) - \frac{1}{2} +\left(-\frac{1}{4} + \frac{k}{4}\right),
\end{equation*}
and so $N_{\mft, 0} = \frac{k-1}{2}$.  Since $V_{\mft,0} \geq V_{\mft,k}$ by Property~\ref{thm:rasmussen-localh}, we have that $N_{\mft,0} = V_{\mft,0}$.   

Next we consider Proposition~\ref{thm:ni-wu} in the case that $\mft$ is self-conjugate on $L(-3,1)$ and $i=1$. From Property \ref{thm:rasmussen-localh}, we have that $V_{\mft,1}$ must be either $\frac{k-1}{2}$ or $\frac{k-3}{2}$.  Since $N_{\mft,1} = \max\{V_{\mft,1},V_{\mft,k-1}\} = V_{\mft,1}$, the same conclusion applies to $N_{\mft,1}$. 

We claim that there is no $\spinc$ structure on $L(-3k, 1)$ compatible with \eqref{ni-wu-formula} and $N_{\mft,1} = \frac{k-1}{2}$ or $\frac{k-3}{2}$. Suppose for contradiction that such a $\spinc$ structure exists.   Denote the corresponding value in $\Z/3k$ by $j$.  Of course, $j \neq 0$, since $j = 0$ is induced by $i = 0$ on $L(k,1)$.  

First, consider the case that $N_{\mft,1} = \frac{k-1}{2}$. Applying \eqref{ni-wu-formula} with $i = 1$ yields
\[
	k-1 =  \left(-\frac{1}{4} + \frac{(2j-3k)^2}{12k} \right) -  \frac{1}{2} + \left( -\frac{1}{4} + \frac{(2-k)^2}{4k} \right),
\]
for some $0 < j < 3k$.  This simplifies to the expression
\[
	k(3j + 3) = j^2 + 3.
\]
Thus $j$ is a positive integral root of the quadratic equation
\[
	f(j) = j^2 - 3kj -(3k-3).
\]
For $k > 0$, there are no integral roots with $0 < j < 3k$. 

Suppose next that $N_{\mft,1} = \frac{k-3}{2}$. Equation \eqref{ni-wu-formula} now yields
\[
	k-3 =  \left(-\frac{1}{4} + \frac{(2j-3k)^2}{12k}\right) -  \frac{1}{2} + \left( -\frac{1}{4} + \frac{(2-k)^2}{4k} \right)
\]
which simplifies to the expression
\[
	k(3j - 3) = j^2 + 3.
\]
Thus $j$ is an integral root of the quadratic equation
\[
	f(j) = j^2 - 3kj +(3k+3).
\]
However, the only integral roots of this equation for $k > 0$ occur when $k=2$ and $j=3$, and we have determined that $k$ is odd.  Thus, we have completed the proof. 
\end{proof}


\begin{proposition}\label{prop:3k-+}
If $k = 1$ or $k>2$, then  $L(-3k,1)$ cannot be obtained by $+k/1$-surgery on a null-homologous knot in $L(3,1)$.   
\end{proposition}

\begin{proof}
As a warning to the reader, this is the unique case where Proposition~\ref{prop:even} does not apply, and we must also allow for the case of $k$ even.  Other than this, the argument mirrors the proof of Proposition~\ref{prop:3k+-} with some extra care to identify the appropriate self-conjugate $\spinc$ structures.    

Consider the statement of Proposition \ref{thm:ni-wu} in the case that $\mft$ is self-conjugate on $L(3,1)$ and $i=0$ on $L(k,1)$. We would like to determine which $\spinc$ structure on $L(-3k,1)$ is induced by \eqref{ni-wu-formula}.  As in the previous cases, when $k$ is odd, $\mft_0$ is the unique self-conjugate $\spinc$ structure on $L(-3k, 1)$, which corresponds to $0$.  We now establish the same conclusion if $k$ is even.  In this case, the proof of Lemma~\ref{lem:spin-3k} shows that the $\spinc$ structures $\mft_0$ and $\mft_{\frac{k}{2}}$, as in Proposition~\ref{thm:ni-wu}, give the two self-conjugate $\spinc$ structures on $L(-3k,1)$.  On the other hand, \eqref{eq:spin-formula} asserts that the numbers $0$ and $3k/2$ also correspond to the two self-conjugate $\spinc$ structures on $L(-3k,1)$.  Proposition~\ref{prop:even} shows that $3k/2$ corresponds to the $\spin$ structure that extends over the two-handle cobordism, while Lemma~\ref{lem:spin-3k}\eqref{spin-3k:2} tells us that $\mft_0$ is the $\spin$ structure that does not extend.  In other words, $\mft_0$ corresponds to $0$ on $L(-3k,1)$.  

Equations \eqref{eq:d-L(n,1)} and \eqref{ni-wu-formula} now yield
\begin{equation*}
	2N_{\mft, 0} = d(L(3k, 1), 0) + d(L(3,1), 0) + d(L(k, 1),0) = \left(-\frac{1}{4} + \frac{3k}{4}\right) + \frac{1}{2} +\left(-\frac{1}{4} + \frac{k}{4}\right),
\end{equation*}
and so $N_{\mft, 0} = \frac{k}{2}$.  Since $V_{\mft,0} \geq V_{\mft,k}$, we have that $N_{\mft,0} = V_{\mft,0}$.   

Next we consider Proposition~\ref{thm:ni-wu} in the case that $\mft$ is self-conjugate on $L(3,1)$ and $i=1$. From Property \ref{thm:rasmussen-localh}, we have that $V_{\mft,1}$ must be either $\frac{k}{2}$ or $\frac{k-2}{2}$.  Since $N_{\mft,1} = \max\{V_{\mft,1},V_{\mft,k-1}\} = V_{\mft,1}$, we also have $N_{\mft,1} = \frac{k}{2}$ or $\frac{k-2}{2}$. 

We claim that there is no $\spinc$ structure on $L(-3k, 1)$ compatible with $N_{\mft,1} = \frac{k}{2}$ or $\frac{k-2}{2}$ in \eqref{ni-wu-formula}. Suppose for the contrary such a $\spinc$ structure exists corresponding to $j \in \Z/3k$.  Again, $j \neq 0$.  

In the case that $N_{\mft,1} = \frac{k}{2}$, then \eqref{ni-wu-formula} yields
\[
	k =  \left(-\frac{1}{4} + \frac{(2j-3k)^2}{12k} \right) +  \frac{1}{2} + \left( -\frac{1}{4} + \frac{(2-k)^2}{4k} \right),
\]
which simplifies to the expression
\[
	k(3j + 3) = j^2 + 3.
\]
As discussed in the proof of Proposition~\ref{prop:3k+-}, there are no integral solutions with $k \geq 1$ and $0 < j < 3k$.  

In the case that $N_{\mft,1} = \frac{k-2}{2}$, then Equations \eqref{eq:d-L(n,1)} and \eqref{ni-wu-formula} now yield
\[
	k-2 =  \left(-\frac{1}{4} + \frac{(2j-3k)^2}{12k}\right) +  \frac{1}{2} + \left( -\frac{1}{4} + \frac{(2-k)^2}{4k} \right),
\]
which simplifies to the expression
\[
	k(3j - 3) = j^2 + 3.
\]
As discussed in the proof of Proposition \ref{prop:3k+-}, there is a unique integral root corresponding to $k = 2$ and $j = 3$.  This exceptional case arises due to the distance one lens space surgery from $L(3,1)$ to $-L(6,1)$ described in \cite[Corollary 1.4]{BakerOPT}\footnote{While this is written as $L(6,1)$ in \cite{BakerOPT}, Baker was working in the unoriented category.} (see also \cite[Table A.5]{MartelliPetronio}).
\end{proof}

\subsection{From $L(3,1)$ to $L(n, 1)$ where $|n|\equiv 1$ (mod 3)}
The goal of this section is to prove the following.

\begin{proposition}\label{prop:3k+1}
There is no distance one surgery from $L(3,1)$ to $L(n,1)$ where $|n| = 3k+1$, except when $n = \pm 1, 4$ or $7$.  
\end{proposition}
 As a preliminary, we use \eqref{eq:d-lens} to explicitly compute the $d$-invariant formulas that will be relevant here.  For $k \geq 0$,   
\begin{align}
\label{eq:d(L(3k+1,1),j)} & d(L(3k+1,1),j)= -\frac{1}{4} + \frac{(-1+2j-3k)^2}{4(3k+1)} \\ 
\label{eq:d(L(3k+1,1),0)} & d(L(3k+1,1),0)= \frac{3k}{4} \\ 
\label{eq:d(L(3k+1,3),1)} & d(L(3k+1,3),1)= \frac{k}{4} \\
\label{eq:d(L(3k+1,3),4)} & d(L(3k+1,3),4)= \frac{8-11k+3k^2}{4(3k+1)}.  
\end{align}  

We will also need the following proposition about the $d$-invariants of surgery, proved in Proposition~\ref{prop:3k+1-surgery-formula} in Section \ref{sec:mappingcone}.  This can be seen as a partial analogue of Proposition~\ref{thm:ni-wu} for homologically essential knots.

\begin{proposition}\label{prop:3k+1-surg}
Let $K$ be a knot in $L(3,1)$.  Suppose that a distance one surgery on $K$ produces an L-space $Y'$ where $|H_1(Y')| = 3k+1$ is odd.  Then there exists a non-negative integer $N_0$ satisfying 
\begin{equation}
\label{eq:3k+1-spin} d(Y',\mft) + d(L(3k+1,3),1) = 2N_0,
\end{equation}
where $\mft$ is the unique self-conjugate $\spinc$ structure on $Y'$.  

Furthermore, if $N_0 \geq 2$, then there exists $\mft' \in \spinc(Y')$ and an integer $N_1$ equal to $N_0$ or $N_0- 1$ satisfying
\begin{equation}
\label{eq:3k+1-nonspin} d(Y',\mft') + d(L(3k+1,3),4) = 2N_1. 
\end{equation}
\end{proposition}

With the above technical result assumed, the proof of Proposition~\ref{prop:3k+1} will now follow quickly. The strategy of proof is similar to that used in the case of $L(3k,1)$. 

\begin{proof}[Proof of Proposition~\ref{prop:3k+1}]
By Lemma~\ref{lem:cobordism} and Proposition~\ref{prop:even}, we see that $n$ must be odd or $n = 4$.  In the latter case, we construct a coherent band surgery from the torus knot $T(2,3)$ to $T(2,4)$ in Figure~\ref{fig:bandings-2}, which lifts to a distance one surgery from $L(3,1)$ to $L(4,1)$.  Therefore, for the remainder of the proof, we assume that $n$ is odd.   We also directly construct a non-coherent band surgery from $T(2,3)$ to $T(2,7)$ and the unknot in Figure~\ref{fig:bandings-1}, so we now focus on ruling out all even values of $k \geq 4$. 

We begin by ruling out distance one surgeries to $+L(3k+1,1)$ with $k \geq 4$. Since $n=3k+1$ is odd, there is a unique self-conjugate $\spinc$ structure on $L(3k+1, 1)$. By Equations \eqref{eq:d(L(3k+1,1),0)}, \eqref{eq:d(L(3k+1,3),1)} and \eqref{eq:3k+1-spin}, we have $N_0 = \frac{k}{2}$.  Since $k$ is at least 4, we have $N_0 \geq 2$.  We claim that there is no solution to Equation \eqref{eq:3k+1-nonspin} with $N_1 = \frac{k}{2}$ or $\frac{k-2}{2}$.  This will complete the proof for the case of $+L(3k+1,1)$.  

First, consider the case of $N_1 = \frac{k}{2}$.  Simplifying Equation \eqref{eq:3k+1-nonspin} as in the proof of Proposition~\ref{prop:3k+-} we obtain
\[
j^2-(1+3k)j+ (2-3k) = 0.  
\]
It is straightforward to see that there are no non-negative integral roots of the quadratic equation for positive $k$.  

Next, we consider $N_1 = \frac{k-2}{2}$.  In this case, \eqref{eq:3k+1-nonspin} implies
\[
j^2-(1+3k)j + (3k+4) = 0.
\]
The roots are of the form
\[
j = \frac{1}{2}(1 + 3k \pm \sqrt{9k^2-6k-15}).  
\]
It is straightforward to verify that for $k \geq 4$, the lesser root is always strictly between 1 and 2, while the greater root is strictly between $3k-1$ and $3k$.  Therefore, there are no integer solutions.  This completes the proof for the case of $+L(3k+1,1)$.  

To complete the proof of Proposition~\ref{prop:3k+1}, it remains to show that $L(-3k+1,1)$, with $k > 0$ even, cannot be obtained from a distance one surgery along a homologically essential knot in $L(3,1)$. Proposition~\ref{prop:3k+1-surg} establishes 
\begin{equation}\label{eq:3k+1-neg}
d(L(3k+1,1),0) = d(L(3k+1,3),1) - 2N_0
\end{equation}  
for some non-negative integer $N_0$.  However, from Equations \eqref{eq:d(L(3k+1,1),0)} and \eqref{eq:d(L(3k+1,3),1)}, we have that
\[
d(L(3k+1,1),0) = \frac{3k}{4} > \frac{k}{4} = d(L(3k+1,3),1),
\]
which contradicts \eqref{eq:3k+1-neg}.  
\end{proof}

\subsection{From $L(3,1)$ to $L(n, 1)$ where $|n|\equiv -1$ (mod 3)}
In this section, we handle the final case in the proof of Theorem~\ref{thm:main}:

\begin{proposition}\label{thm:3k-1}
There is no distance one surgery from $L(3,1)$ to $L(n,1)$, where $|n| = 3k - 1 > 0$, except when $n = \pm 2$.  
\end{proposition}

As before, we state the $d$-invariant formulas that will be relevant for proving this theorem first.  
\begin{align}\label{eq:d(3k-1,3)}
d(L(3k-1,1),i) &= -\frac{1}{4} + \frac{(2i-3k+1)^2}{4(3k-1)} \\
\label{eq:d(3k-1,3)-1} d(L(3k-1,3),1) &= \frac{k-2}{4} \\
\label{eq:d(3k-1,3)-4} d(L(3k-1,3),4) &= \frac{3k^2-19k+18}{4(3k-1)}.  
\end{align}
The above follow easily from \eqref{eq:d-lens}.  

Next, we state a technical result about the $d$-invariants of surgery, similar to Proposition~\ref{prop:3k+1-surg} above, that we will also prove in Proposition~\ref{prop:3k-1-surgery-formula}.

\begin{proposition}\label{prop:3k-1-surgery}
Let $K$ be a knot in $L(3,1)$.  Suppose that a distance one surgery on $K$ produces an L-space $Y'$ where $|H_1(Y')| = 3k-1 > 0$.  Then, there exists a non-negative integer $N_0$ and a self-conjugate $\spinc$ structure $\mft$ on $Y'$ such that 
\begin{equation}
\label{eq:3k-1-spin} d(Y',\mft)  = d(L(3k-1,3),1) - 2N_0.
\end{equation}
In the case that $k$ is odd, if $\mft \neq \tilde{\mft}$ for some self-conjugate $\tilde{\mft}$, then $d(Y',\tilde{\mft}) = \frac{1}{4}$ .

Furthermore, if $N_0 \geq 2$, then there exists another $\spinc$ structure $\mft'$ on $Y'$ and an integer $N_1$ equal to $N_0$ or $N_0 - 1$ satisfying 
\begin{equation}
\label{eq:3k-1-nonspin} d(Y',\mft')  = d(L(3k-1,3),4) - 2N_1. 
\end{equation}
\end{proposition}

With this, the proof of Proposition~\ref{thm:3k-1} will be similar to the previous two cases.  

\begin{proof}[Proof of Proposition~\ref{thm:3k-1}]
In the case that $n = \pm 2$, we may construct a non-coherent banding from $T(2,3)$ to the Hopf link, as shown in Figure~\ref{fig:bandings-2}, which lifts to a distance one surgery from $L(3,1)$ to $L(2,1) \cong L(-2,1)$.  Therefore, we must rule out the case of $n = \pm (3k-1)$ with $k \geq 2$.

The proof will now be handled in two cases, based on the sign of $n$.  First, we suppose that $+L(3k-1,1)$, with $k \geq 2$, is obtained by a distance one surgery on $L(3,1)$.  By Proposition~\ref{prop:even}, we only need to consider the case that $3k-1$ is odd.  Using \eqref{eq:d(3k-1,3)-1}, we compute 
\[
	d(L(3k-1,1),0) = \frac{3k-2}{4} > \frac{k-2}{4} = d(L(3k-1,3),1).
\]
This contradicts Proposition~\ref{prop:3k-1-surgery}.         

Now, we suppose there is a distance one surgery from $L(3,1)$ to $L(-3k-1,1)$ with $k \geq 2$.  By Lemma~\ref{lem:homology-2}, we may assume that $n$ is even.    We begin with the case of $k = 3$.  Lemma~\ref{lem:homology-1} implies that if $-L(8,1)$ was obtained by a distance one surgery, then the linking form of $-L(8,1)$ must be equivalent to $\frac{3}{8}$.  This is impossible since 5 is not a square mod 8.  Thus, we restrict to the case of $ k \geq 5$ for the rest of the proof.    

Proposition~\ref{prop:3k-1-surgery} and the fact that $d(L(-3k-1,1),0) \neq \frac{1}{4}$ imply that 
\[
	-d(L(3k-1,1),0) = d(L(3k-1,3),1) -2N_0
\]
for some non-negative integer $N_0$.  We compute from \eqref{eq:d(3k-1,3)} and \eqref{eq:d(3k-1,3)-1} that 
\[
N_0 = \frac{k-1}{2}.    
\]

Since we are in the case of $k \geq 5$, we may apply \eqref{eq:3k-1-nonspin}.  Combined with \eqref{eq:d(3k-1,3)-4}, this yields
\[
\frac{1}{4} - \frac{(2j-3k+1)^2}{4(3k-1)} = \frac{3k^2 - 19k + 18}{4(3k-1)} - 2N_1, 
\]
for some $0 < j < 3k-1$.  Equivalently, 
\[
N_1 = \frac{5+j + j^2 - 7k - 3jk + 3k^2}{2(3k-1)}.
\]  
Here $N_1 = \frac{k-1}{2}$ or $\frac{k-3}{2}$.  

In the case of $\frac{k-1}{2}$, we are looking for integral roots of the quadratic equation
\[
f(j) = j^2 + j(1-3k) + (4-3k).  
\]
For $k \geq 5$, there are no roots between 0 and $3k-1$.  For the case of $\frac{k-3}{2}$, we are instead looking for integral roots of the quadratic
\[
f(j) = j^2 + j(1-3k) + (3k+2).
\]  
There are no integral roots in this case for $k \geq 5$.  This completes the proof.
\end{proof}

\section{The mapping cone formula and $d$-invariants}
\label{sec:mappingcone}
In this section, we prove the following two key technical statements which were used above in the proof of Theorem~\ref{thm:main} in the cases of $|n| \equiv \pm 1 \pmod{3}$.  These provide analogues of Proposition~\ref{thm:ni-wu} for certain surgeries on homologically essential knots in $L(3,1)$.    

\begin{proposition}\label{prop:3k-1-surgery-formula}
Let $Y = L(3,1)$ and suppose that $Y'$ is an L-space obtained from a distance one surgery on a knot in $Y$, where $|H_1(Y')| = 3k -1$ with $k \geq 1$.  Then there exists a non-negative integer $N_0$ and a self-conjugate $\spinc$ structure $\mft$ on $Y'$ satisfying
\begin{equation}\label{eq:3k-1-0}
d(Y',\mft) = d(L(3k-1,3),1) - 2N_0. 
\end{equation}
Furthermore, if $N_0 \geq 2$, then there exists an integer $N_1$ satisfying $N_0 \geq N_1 \geq N_0 - 1$ and
\begin{equation}\label{eq:3k-1-1}
d(Y',\mft + PD[\mu])  = d(L(3k-1,3),4) - 2N_1.
\end{equation}
Here, $[\mu]$ represents the class in $H_1(Y')$ induced by the meridian of the knot.

Moreover, if $\tilde{\mft} \neq \mft$ for a self-conjugate $\spinc$ structure $\tilde{\mft}$, then $d(Y',\tilde{\mft}) = \frac{1}{4}$.
\end{proposition}

\begin{proposition}\label{prop:3k+1-surgery-formula}
Let $Y = L(3,1)$ and suppose that $Y'$ is an L-space obtained from a distance one surgery on a knot in $Y$, where $|H_1(Y')| = 3k +1$ with $k \geq 0$.  Then there exists a non-negative integer $N_0$ and a self-conjugate $\spinc$ structure $\mft$ on $Y'$ satisfying
\begin{equation}\label{eq:3k+1-0}
d(Y',\mft) + d(L(3k+1,3),1) = 2N_0. 
\end{equation}
Furthermore, if $N_0 \geq 2$, then there exists an integer $N_1$ satisfying $N_0 \geq N_1 \geq N_0 - 1$ and
\begin{equation}\label{eq:3k+1-1}
d(Y',\mft + PD[\mu])  + d(L(3k+1,3),4) = 2N_1.
\end{equation}
Here, $[\mu]$ represents the class in $H_1(Y')$ induced by the meridian of the knot.

Moreover, if $\tilde{\mft} \neq \mft$ for a self-conjugate $\spinc$ structure $\tilde{\mft}$, then $d(Y',\tilde{\mft}) = \frac{3}{4}$.
\end{proposition}

\begin{remark}
We expect that the conclusions of these two propositions hold independently of $Y'$ being an L-space and the value of $N_0$.    
\end{remark}

The general argument for the above propositions is now standard and is well-known to experts.  The strategy is to study the $d$-invariants using the mapping cone formula for rationally null-homologous knots due to Ozsv\'ath-Szab\'o \cite{OSRational}.  In Section~\ref{subsec:surgery-review}, we review the mapping cone formula.  In Sections~\ref{sec:mc-preliminaries}, ~\ref{sec:mc-spinc} and ~\ref{sec:mc-truncation} we establish certain technical results about the mapping cone formula analogous to properties well-known for knots in $S^3$.  Finally, in Section~\ref{sec:mc-proof}, we prove Propositions~\ref{prop:3k-1-surgery-formula} and \ref{prop:3k+1-surgery-formula}.

\subsection{The mapping cone for rationally nullhomologous knots}\label{subsec:surgery-review}
In this subsection, we review the mapping cone formula from \cite{OSRational}, which will allow us to compute the Heegaard Floer homology of distance one surgeries on knots in a rational homology sphere.  We assume the reader is familiar with the knot Floer complex for knots in $S^3$; we will use standard notation from that realm.  For simplicity, we work in the setting of a rational homology sphere $Y$.  (As a warning, $Y$ will be $-L(3,1)$ when proving Proposition~\ref{prop:3k+1-surgery-formula}.) All Heegaard Floer homology computations will be done with coefficients in $\FF = \Z/2$.  As mentioned previously, singular homology groups are assumed to have coefficients in $\Z$, unless otherwise noted.

Choose an oriented knot $K \subset Y$ with meridian $\mu$ and a framing curve $\lambda$, i.e. a slope $\lambda$ on the boundary of a tubular neighborhood of $K$ which intersects the meridian $\mu$ once transversely.  Here, $\lambda$ naturally inherits an orientation from $K$.  
Let $Y'$ denote the result of $\lambda$-surgery.  

We write $\relspinc(Y,K)$ for the relative $\spinc$ structures on $(M,\partial M)$, which has an affine identification with $H^2(Y,K) = H^2(M,\partial M)$.  Here, $M = Y -\mathcal{N}(K)$. If $K$ generates $H_1(Y)$, then $\relspinc(Y,K)$ is affinely isomorphic to $\Z$.  
In our applications, this will be the case.

There exist maps $G_{Y,\pm K}: \relspinc(Y,K) \to \spinc(Y)$ satisfying 
\begin{equation}\label{eq:G-equivariance}
G_{Y,\pm K}(\xi + \kappa) = G_{Y, \pm K}(\xi) + i^* \kappa,
\end{equation} 
where $\kappa \in H^2(Y,K)$ and $i: (Y, pt) \to (Y,K)$ is inclusion.  Here, $-K$ denotes $K$ with the opposite orientation.  We have  
\[
G_{Y,-K}(\xi) = G_{Y,K}(\xi) + PD[\lambda].
\]
If $Y' = Y_\lambda(K)$ is obtained by surgery on $K$, we will write $K'$ or $K_\lambda$ for the core of surgery.  

Associated to $\xi \in \relspinc(Y,K)$ is the $\Z \oplus \Z$-filtered knot Floer complex $C_\xi  = CFK^\infty(Y,K,\xi)$.  Here, the bifiltration is written (\emph{algebraic}, \emph{Alexander}).  
We have $C_{\xi + PD[\mu]} = C_\xi[(0,-1)]$, i.e. we shift the Alexander filtration on $C_\xi$ by one.  Note that not every relative $\spinc$ structure is necessarily related by a multiple of $PD[\mu]$, so we are not able to use this to directly compare the knot Floer complexes for an arbitrary pair of relative $\spinc$ structures.  

For each $\xi \in \relspinc(Y,K)$, we define the complexes $\Aplus_\xi = C_\xi\{\max\{i,j\} \geq 0\}$ and $\Bplus_\xi = C_\xi\{i \geq 0\}$.  The complex $\Bplus_\xi$ is simply $CF^+(Y,G_{Y,K}(\xi))$, while $\Aplus_\xi$ represents the Heegaard Floer homology of a large surgery on $K$ in a certain $\spinc$ structure, described in slightly more detail below.  

The complexes $\Aplus_\xi$ and $\Bplus_\xi$ are related by grading homogenous maps 
$$
\vplus_\xi: \Aplus_\xi \to \Bplus_\xi, \ \ \hplus_\xi : \Aplus_\xi \to \Bplus_{\xi + PD[\lambda]}.
$$
Rather than defining these maps explicitly, we explain how these can be identified with certain cobordism maps as follows.  Fix $n \gg 0$ and consider the three-manifold $Y_{n\mu+\lambda}(K)$ and the induced cobordism from $Y_{n\mu+\lambda}(K)$ to $Y$ obtained by attaching a two-handle to $Y$, reversing orientation, and turning the cobordism upside down.  We call this cobordism $W'_n$, which is negative-definite.  Fix a generator $[F] \in H_2(W'_n,Y)$ such that $PD[F]|_Y = PD[K]$.  Equip $Y_{n\mu+\lambda}(K)$ with a $\spinc$ structure $\mft$.  It is shown in \cite[Theorem 4.1]{OSRational} that there exist two particular $\spinc$ structures $\mfv$ and $\mfh= \mfv + PD[F]$ on $W'_n$ which extend $\mft$ over $W'_n$ and an association $\Xi:\spinc(Y_{n\mu+\lambda}(K)) \to \relspinc(Y,K)$ satisfying commutative squares:
\begin{equation}\label{eq:large-surg}
\xymatrix{
 CF^+(Y_{n\mu+\lambda}(K),\mft)  \ar[d]^{f_{W'_n,\mfv}}  \ar[r]^-\simeq  & \Aplus_\xi  \ar[d]^{\vplus_\xi} & CF^+(Y_{n\mu + \lambda}(K),\mft)   \ar[d]^{f_{W'_n,\mfh}}  \ar[r]^-\simeq &  \Aplus_\xi \ar[d]^{\hplus_\xi} \\
 CF^+(Y,G_{Y,K}(\xi))   \ar[r]^-\simeq  &\Bplus_\xi  & CF^+(Y,G_{Y,-K}(\xi))  \ar[r]^-\simeq & \Bplus_{\xi + PD[\lambda]},
}
\end{equation}
where $\xi = \Xi(\mft)$.  Here, $f_{W'_n,\mfs}$ denotes the $\spinc$ cobordism map in Heegaard Floer homology, as defined in \cite{OSSmoothFour}.

More generally, there exists a map $E_{K,n,\lambda} : \spinc(W'_n) \to \relspinc(Y,K)$ such that if $\mfv$ and $\mfh$ are as above, then 
\begin{equation}\label{eq:E-PDS}
E_{K,n,\lambda}(\mfv) = \xi, \ E_{K,n,\lambda}(\mfh) = \xi + n PD[\mu] + PD[\lambda].  
\end{equation}
To make the notation more suggestive, we will write $\mfv_\xi$ and $\mfh_\xi$ for the associated $\spinc$ structures on $W'_n$ appearing in \eqref{eq:large-surg}.

Recall that for any $\spinc$ rational homology sphere, the Heegaard Floer homology contains a distinguished submodule isomorphic to $\T^+ = \FF[U,U^{-1}]/U \cdot \FF[U]$, called the {\em tower}.  Since $W'_n$ is negative-definite, on the level of homology, $\vplus_\xi$ induces a grading homogeneous non-zero map between the towers, which is necessarily multiplication by $U^N$ for some integer $N \geq 0$.  We denote this integer by $V_\xi$.  The integer $H_{\xi}$ is defined similarly.   These numbers $V_\xi$ are also known as the local $h$-invariants, originally due to Rasmussen \cite{RasmussenThesis}.  A direct analogue of \cite[Proposition 7.6]{RasmussenThesis} (Property~\ref{thm:rasmussen-localh} above), using $C_{\xi + PD[\mu]} = C_{\xi}[(0,-1)]$, shows that for each $\xi \in \relspinc(Y,K)$, 
\begin{equation}\label{eq:Vs-inequality}
V_\xi \geq V_{\xi + PD[\mu]} \geq V_\xi - 1.
\end{equation}  

We are now ready to define the mapping cone formula.  Define the map 
\begin{equation}\label{eq:mapping-cone}
\Phi: \bigoplus_\xi \Aplus_\xi \to \bigoplus_\xi \Bplus_\xi, \ (\xi,a) \mapsto (\xi, \vplus_\xi(a)) + (\xi + PD[\lambda], \hplus_\xi(a)),
\end{equation}
where the first component of $(\xi, a)$ simply indicates the summand in which the element lives.  Notice that the mapping cone of $\Phi$ splits over equivalence classes of relative $\spinc$ structures, where two relative $\spinc$ structures are equivalent if they differ by an integral multiple of $PD[\lambda]$.  We let the summand of the cone of $\Phi$ corresponding to the equivalence class of $\xi$ be written $\Xplus_\xi$.  Ozsv\'ath and Szab\'o show that there exist grading shifts on the complexes $\Aplus_\xi$ and $\Bplus_\xi$ such that $\Xplus_\xi$ can be given a consistent relative $\Z$-grading \cite{OSRational}.  In fact, these shifts can be done to $\Xplus_\xi$ with an absolute $\mathbb{Q}$-grading.  While we do not describe the grading shifts explicitly at the present moment, it is important to point out that these shifts only depend on the homology class of the knot.  With this, we are ready to state the connection between the mapping cone formula and surgeries on $K$.  

\begin{theorem}[Ozsv\'ath-Szab\'o, \cite{OSRational}]\label{thm:mappingcone}
Let $\xi \in \relspinc(Y,K)$.  Then there exists a quasi-isomorphism of absolutely-graded $\FF[U]$-modules, 
\begin{equation}\label{eq:mappingcone}
\Xplus_\xi \simeq CF^+(Y_\lambda(K), G_{Y_\lambda(K),K_\lambda}(\xi)). \\
\end{equation}
\end{theorem}

Finally, we remark that the entire story above has an analogue for the hat flavor of Heegaard Floer homology.  We denote the objects in the hat flavor by $\Ahat_\xi, \Xhat_\xi, \vhat_\xi$, etc.  The analogue of \eqref{eq:mappingcone} is then a quasi-isomorphism
\begin{equation}\label{eq:mappingcone-hat}
\Xhat_\xi \simeq \widehat{CF}( Y_\lambda(K), G_{Y_\lambda(K),K_\lambda}(\xi)).
\end{equation}

\subsection{Preliminaries specific to knots in $L(3,1)$}\label{sec:mc-preliminaries}
Through Sections~\ref{sec:mc-preliminaries}-\ref{sec:mc-truncation}, $K$ will denote a homologically essential knot in $Y = L(3,1)$ and $\lambda$ will denote a framing such that $Y' = Y_\lambda(K)$ is an L-space with $|H_1(Y')| = 3k-1$ for some $k > 0$.  The case of $|H_1(Y')| = 3k+1$ is dealt with similarly, and the necessary changes are described in Section~\ref{sec:mc-proof} below.  Recall that we give $\lambda$ the orientation induced by $K$. 

The mapping cone formula for any homologically essential knot in $L(3,1)$ is easier to describe than in generality.  We have that $\relspinc(Y,K) \cong \Z$.  Write $[m]$ for the generator of $H_1(M)$ such that $[\mu] = 3[m]$ (instead of $-3[m]$).  Consequently, since $[\mu] \cdot [\lambda] = 1$, we have that $[\lambda] = (3k-1) [m]$ by Lemma~\ref{lem:homology-1}\eqref{homology-1-3}.  Therefore, for fixed $\xi \in \relspinc(Y,K)$, we see that the mapping cone $\Xplus_\xi$ consists of the $A_{\xi'}$ and $B_{\xi'}$ where $\xi' - \xi = (3k-1)j \cdot PD[m]$ for some $j \in \Z$.  For a more pictorial representation, see Figure~\ref{fig:mapping-cone-picture} for the case of $k = 2$.  

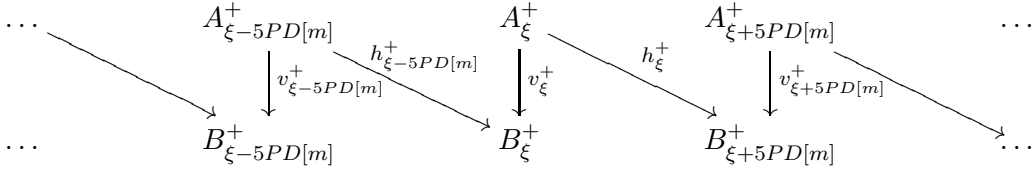
\begin{figure}
\[
\xymatrix{
\ldots \ar[drr] & &  \Aplus_{\xi - 5PD[m]} \ar[d]^{\vplus_{\xi - 5PD[m]}}  \ar[drr]^{\hplus_{\xi - 5PD[m]}} & & \Aplus_\xi \ar[d]^{\vplus_\xi} \ar[drr]^{\hplus_\xi} & &\Aplus_{\xi + 5 PD[m]} \ar[d]^{\vplus_{\xi + 5PD[m]}} \ar[drr] & & \ldots \\
\ldots & & \Bplus_{\xi - 5PD[m]} & & \Bplus_\xi & & \Bplus_{\xi + 5PD[m]} & & \ldots 
}
\]
\caption{The mapping cone formula for surgery on a knot in $L(3,1)$ resulting in a three-manifold $Y'$ with $|H_1(Y')| = 5$ corresponding to the $\spinc$ structure $G_{Y',K'}(\xi)$.  } \label{fig:mapping-cone-picture}
\end{figure}

Ozsv\'ath and Szab\'o show that for fixed $\xi$, there exists $N$ such that $\vplus_{\xi+ j\cdot PD[\mu]}$ and $\hplus_{\xi - j \cdot PD[\mu]}$ are quasi-isomorphisms for $j > N$.  Using this, the mapping cone formula is quasi-isomorphic (via projection) to the quotient complex depicted in Figure~\ref{fig:mapping-cone-truncated}.  We will denote the truncated complex by $\XplusN_\xi$, which now depends on $\xi$, even though the homology does not.  Note that the shape of the truncation is special to the case that $H_1(Y_\lambda(K))$ has order $3k-1$. Were the order to be $3k + 1$, there would instead be one more $\Bplus_\xi$ than $\Aplus_\xi$ and $\hplus_\xi$ would translate by $-(3k+1)PD[m]$.  This issue will be dealt with in Proposition~\ref{prop:3k+1-surgery-formula} by reversing orientations and performing surgery on knots in $-L(3,1)$ instead.

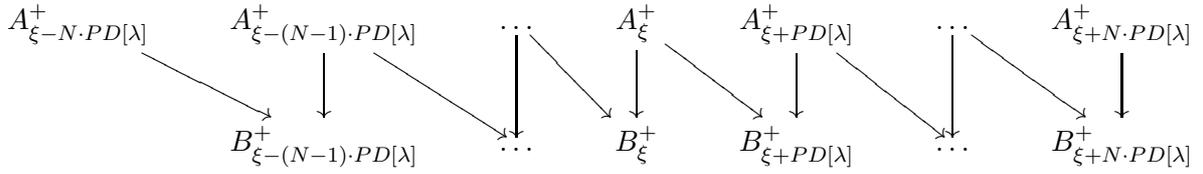
\begin{figure}
\[
\xymatrix{
\Aplus_{\xi - N \cdot PD [\lambda]} \ar[dr] & \Aplus_{\xi - (N-1) \cdot PD [\lambda]} \ar[d] \ar[dr] & \ldots \ar[d] \ar[dr] & \Aplus_{\xi} \ar[d] \ar[dr] & \Aplus_{\xi+PD[\lambda]} \ar[d] \ar[dr] & \ldots \ar[d] \ar[dr] & \Aplus_{\xi + N \cdot PD [\lambda]} \ar[d] \\
& \Bplus_{\xi - (N-1) \cdot PD [\lambda]} & \ldots & \Bplus_{\xi} &\Bplus_{\xi+PD[\lambda]} & \ldots &  \Bplus_{\xi + N \cdot PD [\lambda]}
}
\]
\caption{The truncated mapping cone $\X^{+,N}_\xi$ computing $CF^+(Y_\lambda(K), G_{Y_\lambda(K),K_\lambda}(\xi))$ in the case that $|H_1(Y_\lambda(K))| \equiv -1 \pmod{3}$.} \label{fig:mapping-cone-truncated}
\end{figure}

By \cite[Lemma 6.7]{BBCW}, since $Y_\lambda(K)$ is an L-space obtained by a distance one surgery in an L-space, we have that 
\begin{equation}\label{eq:Axi=T}
H_*(\Ahat_\xi) \cong \FF, \ H_*(\Aplus_\xi) \cong \T^+ \text{ for all } \xi \in \relspinc(Y,K).
\end{equation}
Indeed, the orientation conventions from \cite[Lemma 6.7]{BBCW} are specified by the condition that $[\mu]$ and $[\lambda]$ are positive multiples of the same homology class, which is the setting we are in.  Of course, since $Y = L(3,1)$ is an L-space, we also have that $H_*(\Bhat_\xi) \cong \FF$ and $H_*(\Bplus_\xi)\cong \T^+$ for all $\xi$.  Equation~\eqref{eq:Axi=T} implies that the Heegaard Floer homology of $Y_\lambda(K)$ is completely determined by the numbers $V_\xi$ and $H_\xi$ for each $\xi \in \relspinc(Y,K)$.

\subsection{$\spinc$ structures}\label{sec:mc-spinc}
In order to understand the Heegaard Floer homology of surgery using the mapping cone, we must understand the various $\spinc$ and relative $\spinc$ structures that appear.  These are well-understood in the setting of a nullhomologous knot, and are likely known to experts, but we include them here for completeness.  As in the previous subsection, $K$ will denote a homologically essential knot in $Y = L(3,1)$ and $\lambda$ is a framing such that $Y' = Y_\lambda(K)$ is an L-space with $|H_1(Y')| = 3k-1$ for some $k > 0$. 

Fix $n \gg 0$ throughout.  By fixing the appropriate parity of $n$, we can compute from Lemma~\ref{lem:homology-1} that ``large positive surgery'', i.e. $Y_{n\mu+\lambda}(K)$, has a unique self-conjugate $\spinc$ structure.  We denote this by $\mft_0$.  Further, let $\xi_0 = \Xi(\mft_0)$ be the induced relative $\spinc$ structure as in \eqref{eq:large-surg}.  Recall that for $\xi \in \relspinc(Y,K)$, we write $\mfv_\xi$ and $\mfh_\xi$ to be the $\spinc$ structures on $W'_n$ defined above \eqref{eq:large-surg}.  

\begin{proposition}\label{prop:Vs=H-s}
Let $[\gamma] \in H_1(M)$.  Then, $V_{\xi_0 + PD[\gamma]} = H_{\xi_0 - PD[\gamma]}$. 
\end{proposition}
This is the analogue of the more familiar formula $V_s = H_{-s}$ for knots in $S^3$.  
\begin{proof}
We will use an observation of Ni and Vafaee from \cite[Proof of Lemma 2.6]{NiVafaee}.  Consider the pair $(W'_n, H)$, where $H$ is the 2-handle attached to $Y \times I$.  Note that $H$ is contractible, so we see that $H^2(W'_n)\cong H^2(W'_n,H) \cong H^2(Y,K)$.   By excision, we now see that $H^2(W'_n)$ is naturally identified with $H^2(M, \partial M)\cong\Z$.  We define $\epsilon$ to be this identification.  The assignment $E_{K,n,\lambda}:\spinc(W'_n) \to \relspinc(Y,K)$ discussed above \eqref{eq:E-PDS} is affine over $\epsilon$, i.e., $E_{K,n,\lambda}(\mfs) - E_{K,n,\lambda}(\mfs') = \epsilon(\mfs - \mfs')$.  It follows from \eqref{eq:E-PDS} that $\epsilon(PD[F]) = n PD[\mu] + PD[\lambda]$.  For shorthand, we write $E$ for $E_{K,n,\lambda}$.  

By the conjugation invariance of $\spinc$ cobordism maps in Floer homology \cite[Theorem 3.6]{OSSmoothFour}, it suffices to show that $\mfv_{\xi_0 + PD[\gamma]}$ and $\mfh_{\xi_0 - PD[\gamma]}$ are conjugate $\spinc$ structures on $W'_n$. Because $W'_n$ is definite and $H^2(W'_n)\cong\Z$, the $\spinc$-conjugation classes are completely determined by $c_1^2$.  First, we will establish that $\overline{\mfv}_{\xi_0} = \mfh_{\xi_0}$, i.e., the case of $[\gamma] = 0$.  

It follows from \cite[Proof of Proposition 4.2]{OSRational} that $\mfv_{\xi_0}$ and $\mfh_{\xi_0}$ are characterized as the two $\spinc$ structures on the negative-definite cobordism $W'_n$ extending $\mft_0$ which have the largest values of $c_1^2$.  Indeed, there it is shown that every $\spinc$ structure extending $\mft_0$ is of the form $\mfv_{\xi_0} + n \cdot PD[F]$ and that one of $\mfv_{\xi_0}, \mfh_{\xi_0}$ maximizes the quadratic function $c_1(\mfv_{\xi_0} + n \cdot PD[F])^2$.  If there was an additional $\spinc$ structure sharing the same value of $c_1^2$ with one of $\mfv_{\xi_0}$ or $\mfh_{\xi_0}$, this would imply that the first Chern class of the maximizing $\spinc$ structure would be 0, forcing $W'_n$ to be $\spin$.  By Lemma~\ref{lem:cobordism}, this implies that $|H_1(Y_{n\mu + \lambda}(K))|$ is even, contradicting the choice of $n$ made at the beginning of this subsection.  

Of course $c_1(\mfv_{\xi_0})^2 = c_1(\overline{\mfv}_{\xi_0})^2$ and similarly for $\mfh_{\xi_0}$.  Because $\mft_0$ is self-conjugate on $Y_{n\mu + \lambda}(K)$, we deduce that either $\overline{\mfv}_{\xi_0} = \mfh_{\xi_0}$ and $\overline{\mfh}_{\xi_0} = \mfv_{\xi_0}$ or $\overline{\mfv}_{\xi_0} = \mfv_{\xi_0}$ and $\overline{\mfh}_{\xi_0} = \mfh_{\xi_0}$.   Since $\mfh_{\xi_0} = \mfv_{\xi_0} + PD[F]$, it must be that $\overline{\mfv}_{\xi_0} = \mfh_{\xi_0}$, proving the desired claim for $PD[\gamma] = 0$.  

Now, fix an arbitrary $[\gamma] \in H_1(M)$.  We see that 
\begin{align*}
E(\mfh_{\xi_0 - PD[\gamma]}) &= E(\mfv_{\xi_0 - PD[\gamma]}) + n PD[\mu] + PD[\lambda] \\
&= \xi_0 - PD[\gamma] + n PD[\mu] + PD[\lambda] \\
&= E(\mfh_{\xi_0}) - PD[\gamma] \\
&= E(\mfh_{\xi_0} - \epsilon^{-1}(PD[\gamma])) \\
&= E(\overline{\mfv_{\xi_0} + \epsilon^{-1}(PD[\gamma])}),
\end{align*}
where the first three lines follow from \eqref{eq:E-PDS}, the fourth is the affine action of $H^2(W_n')$ on $\spinc(W_n')$, and the fifth is because $\overline{\mfv}_{\xi_0} = \mfh_{\xi_0}$. 
Since $E$ is injective, we see that $\mfh_{\xi_0 - PD[\gamma]} = \overline{\mfv_{\xi_0} + \epsilon^{-1}(PD[\gamma])}$.  
On the other hand, $E(\mfv_{\xi_0} + \epsilon^{-1}(PD[\gamma])) = E(\mfv_{\xi_0 + PD[\gamma]})$ because $E$ is affine over $\epsilon$, and thus $\mfv_{\xi_0} + \epsilon^{-1}(PD[\gamma]) = \mfv_{\xi_0 + PD[\gamma]}$.  This establishes that $\mfv_{\xi_0 + PD[\gamma]}$ and $\mfh_{\xi_0 - PD[\gamma]}$ are conjugate, which is what we needed to show.  
\end{proof}


\begin{remark}
It follows from the proof of Proposition~\ref{prop:Vs=H-s} that if $\mft_+, \mft_- \in \spinc(Y_{n\mu + \lambda}(K))$ are such that $\Xi(\mft_\pm) = \xi_0 \pm PD[\gamma]$ for $[\gamma] \in H_1(M)$, then $\mft_+$ and $\mft_-$ are conjugate.   
\end{remark}
%

\begin{figure}
\begin{center}
\begin{tikzpicture}

\node[anchor=south west,inner sep=0] at (0,0) {\includegraphics[width=3in]{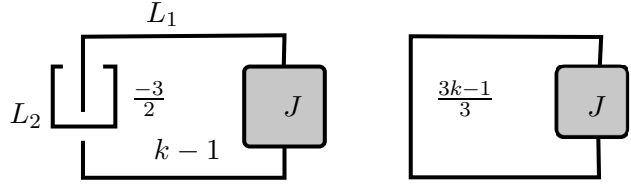}};

\node[label=above right:$L_1$] at (1,1.8){};
\node[label=above right:$L_2$] at (-.8,0.45){};
\node[label=above right:{$J$}] at (2.8,0.6){};
\node[label=above right:{$\frac{-3}{2}$}] at (0.8,0.6){};
\node[label=above right:{$k-1$}] at (1.1,0){};
\node[label=above right:{$J$}] at (6.8,0.6){};
\node[label=above right:{$\frac{3k-1}{3}$}] at (4.8,0.6){};

\end{tikzpicture}
\end{center}

\caption{Surgery on the link $L=L_1\cup L_2\subset S^3$ is equivalent by a slam-dunk move to surgery along the knot $J\subset S^3$. (Left) The surgery diagram also shows integral surgery on the knot $K_J$ in $L(3,1)$ yielding a manifold with $|H_1| = 3k-1$.}
\label{fig:surgerydiagram}
\end{figure}

In order to prove Proposition~\ref{prop:3k-1-surgery-formula}, we will need to identify self-conjugate $\spinc$ structures on $Y_\lambda(K)$ in the mapping cone formula.  This will be done in Lemmas~\ref{lem:spin-surg} and  \ref{lem:theta} below.  Before doing so, it will be useful to describe a particular example of $Y_\lambda(K)$ by a concrete surgery diagram.  (See Figure \ref{fig:surgerydiagram}.) Let $L^J = L_1 \cup L_2$ denote the Hopf link connect sum with a knot $J \subset S^3$ at $L_1$.  We may consider $Y$ as $-3/2$-surgery on $L_2$, where $K$ is the image of $L_1$ under the surgery.  We will write this special knot in $L(3,1)$ as $K_J$. 
In this case, $\lambda$ is represented by the framing $k-1$ on $L_1$, and after a slam-dunk move, we see that the resulting manifold is $S^3_{\frac{3k-1}{3}}(J)$. In general, to compute $p/q$-surgery on a knot $J$ in $S^3$ using the mapping cone formula, we follow the recipe of Ozsv\'ath-Szab\'o.  First, define
$$
\frac{r}{q} =\frac{p}{q} - \left\lfloor \frac{p}{q} \right\rfloor .
$$
Then, consider surgery on the link $L^J$ where $L_2$ has coefficient $-q/r$ and $L_1$ has integral surgery coefficient $\lfloor \frac{p}{q} \rfloor$. 

In particular, $K_U$ will be an important knot to understand later on, where $U$ is the unknot in $S^3$.  This is the reason the $d$-invariants of $L(3k-1,3)$ show up in Proposition~\ref{prop:3k-1-surgery-formula}.  Finally, we note that $K_U$ is a core of the genus one Heegaard splitting of $L(3,1)$.   
\begin{lemma}\label{lem:spin-surg}
Let $\xi_0$ be as above.  Then, $G_{Y_\lambda(K),K_\lambda}(\xi_0)$ is a self-conjugate $\spinc$ structure on $Y_\lambda(K)$.    
\end{lemma}
\begin{proof}
By assumption, $\xi_0 = \Xi(\mft_0)$ is the relative $\spinc$ structure induced by the unique self-conjugate $\spinc$ structure  on the large positive surgery $Y_{n\mu+\lambda}(K)$. Since the statement is purely homological, it suffices to prove the lemma in the case of a particular \emph{model knot}, provided that this knot is homologically essential in $L(3,1)$. 
Thus we consider the model knot $K_U$ as described above.  In our case, we are interested in the $+(k-1)$-framed two-handle attachment along $K_U \subset -L(3,2)$ illustrated in Figure \ref{fig:surgerydiagram}.  The mapping cone formula in this case has been explicitly computed in \cite{OSRational} and can be completely rephrased in terms of the knot Floer complex for the unknot in $S^3$.  More precisely, this is the mapping cone formula for $(3k-1)/3$-surgery along the unknot in $S^3$.


Write $\mathcal{A}^+_s$ and $\mathcal{V}_s$, $\mathcal{H}_s$ for the $A^+_s$-complexes and numerical invariants $V_s, H_s$ coming from the mapping cone formula for integer surgeries along the unknot in $S^3$, computed in \cite[Section 2.6]{OSInteger}.  The proof of \cite[Theorem 1.1]{OSRational} shows that there exists an affine isomorphism $g : \relspinc(Y,K_U) \to \Z$,  such that
\begin{equation}
\Aplus_{\xi} = \mathcal{A}^+_{\lfloor \frac{g(\xi)}{3} \rfloor},  
\end{equation}
for each $\xi \in \relspinc(Y,K_U)$.  Furthermore, we have that $V_\xi = \mathcal{V}_{\lfloor \frac{g(\xi)}{3} \rfloor}$ and $H_\xi = \mathcal{H}_{\lfloor \frac{g(\xi)}{3} \rfloor}$.  In this setting, the $\spinc$ structure $G_{Y_\lambda(K),K_\lambda}(\xi)$ is, up to conjugation, the $\spinc$ structure on $L(3k-1,3)$ corresponding to $g(\xi)$ modulo $3k-1$.   We claim that $g(\xi_0) = 1$, which is sufficient since on $L(3k-1,3)$, 1 corresponds to a self-conjugate $\spinc$ structure by \eqref{eq:spin-formula}.  

From \cite[Section 2.6]{OSInteger}, we have 
\begin{equation}\label{eq:Vs-unknot}
\mathcal{V}_s = \begin{cases} 0 & \text{ if } s \geq 0 \\ -s & \text{ if } s < 0\end{cases}, \mathcal{H}_s = \begin{cases} s & \text{ if } s \geq 0 \\ 0 & \text{ if } s < 0.  \end{cases}
\end{equation}
In order for the $\mathcal{V}_s$ and $\mathcal{H}_s$ to be compatible with Proposition~\ref{prop:Vs=H-s} and \eqref{eq:Vs-inequality}, since $\xi \mapsto \lfloor \frac{g(\xi)}{3} \rfloor$, we must have that $g(\xi_0) = 1$, completing the proof.     
\end{proof}

\subsection{L-space surgeries and truncation}\label{sec:mc-truncation}
We make the same hypotheses on $K \subset Y = L(3,1)$ as in the previous two subsections.  While \eqref{eq:Axi=T} states that $H_*(\Xplus_{\xi}) \cong H_*(\Aplus_\xi) \cong  \T^+$ for each $\xi$ in the case that $Y_\lambda(K)$ is an L-space, we have not determined ``where'' in the mapping cone the non-zero element of lowest grading is supported.  The analogous question is well-known for surgery on knots in $S^3$ (see \cite{NiWu} for example), but is more subtle in the present setting, since we cannot directly compare $V_\xi$ and $V_{\xi + PD[\lambda]}$.  Indeed, $\xi$ and $\xi + PD[\lambda]$ do not differ by a multiple of $PD[\mu]$.  The next lemma will help us to understand this in the case of $[\xi_0]$.  Before stating the lemma, observe that there are natural quotient maps 
\[
\Piplus_{\xi} : \Xplus_{\xi} \to \Aplus_\xi, \ \Pihat_{\xi} : \Xhat_{\xi} \to \Ahat_\xi,
\] 
for any $\xi \in \relspinc(Y,K)$.

\begin{lemma}\label{lem:proj-qi}
Suppose that $Y_\lambda(K)$ is an L-space.  Then, the projection $\Piplus_{\xi_0}: \Xplus_{\xi_0} \to \Aplus_{\xi_0}$ is a quasi-isomorphism.
\end{lemma}
\begin{proof}
It suffices to prove that $\Pihat_{\xi_0}$ is a quasi-isomorphism.  

First, suppose that $V_{\xi_0} > 0$.  Since $H_*(\Aplus_{\xi_0}) = \T^+$, this is equivalent to $\vhat_{\xi_0}$ vanishing on homology.  By Proposition~\ref{prop:Vs=H-s}, we see that $\hhat_{\xi_0}$ vanishes on homology as well.  For notation, write $Q = \ker(\Pihat_{\xi_0})$.  Using the exact triangle
\[
\xymatrix{
H_*(Q) \ar[rr] & & \ar[dl]^{(\Pihat_{\xi_0})_*}  H_*(\Xhat_{\xi_0}) \\
& H_*(\Ahat_{\xi_0}) \ar[ul]^{(\vhat_{\xi_0} + \hhat_{\xi_0})_*} & 
}
\]
we see that $\Pihat_{\xi_0}$ is surjective on homology.  Since $H_*(\Xhat_{\xi_0})$ and $H_*(\Ahat_{\xi_0})$ are both one-dimensional, we see that $\Pihat_{\xi_0}$ must be a quasi-isomorphism.  

Next, suppose that $\vhat_{\xi_0}$ and $\hhat_{\xi_0}$ are non-zero on homology.  Recall that the quotient from $\Xhat_{\xi_0}$ to the truncated complex $\Xhat_{\xi_0}^N$ described above is a quasi-isomorphism.  Therefore, we will show that the projection from $\Xhat^N_{\xi_0}$ to $\Ahat_{\xi_0}$ is a quasi-isomorphism.  Note that the kernel of the quotient from $\Xhat_{\xi_0}^N$ to $\Ahat_{\xi_0}$ is a sum of two complexes, $Q_+$ and $Q_-$, where
\begin{eqnarray*}\label{eqn:Qplusminus}
	Q_+ &:=& \bigoplus_{j=1}^{N}  \Ahat_{\xi_0+j\cdot PD[\lambda]} \oplus \bigoplus_{j=1}^{N}  \Bhat_{\xi_0 + (j)\cdot PD[\lambda]} \\
	Q_- &:=& \bigoplus_{j=-N}^{-1}  \Ahat_{\xi_0+j\cdot PD[\lambda]} \oplus \bigoplus_{j=-N}^{-1}  \Bhat_{\xi_0 + (j+1)\cdot PD[\lambda]}. 
\end{eqnarray*}
These complexes are shown in Figure~\ref{fig:Q}.  
\begin{figure}
\[
\xymatrix{
\Ahat_{\xi_0 - N \cdot PD [\lambda]} \ar[dr]_h & \ldots \ar[d]_v \ar[dr]_h & \Ahat_{\xi_0 - PD[\lambda]} \ar[dr]^h \ar[d]_v & & \Ahat_{\xi_0+PD[\lambda]} \ar[d]_v \ar[dr]_h & \ldots \ar[d]_v \ar[dr]_h & \Ahat_{\xi _0+ N \cdot PD [\lambda]} \ar[d]^v \\
& \ldots & \Bhat_{\xi_0 - PD[\lambda]} & \Bhat_{\xi_0} &\Bhat_{\xi_0+PD[\lambda]} & \ldots &  \Bhat_{\xi_0 + N \cdot PD [\lambda]}
}
\]
\caption{The kernel of the quotient from $\Xhat_{\xi_0}^N$ to $\Ahat_{\xi_0}$.  The left summand (respectively right summand) corresponds to $Q_-$ (respectively $Q_+$).} \label{fig:Q}
\end{figure}
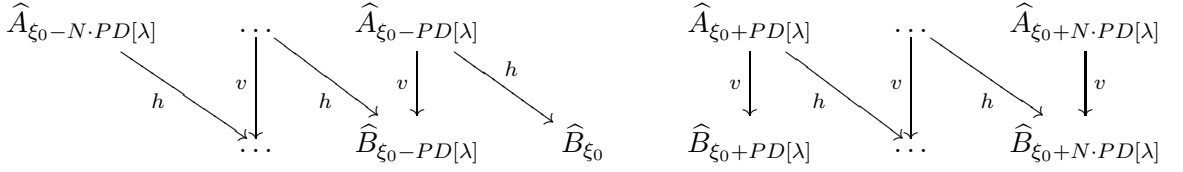	
Note that $\vhat_{\xi_0+j\cdot PD[\lambda]}$ (respectively $\hhat_{\xi_0+j\cdot PD[\lambda]}$) is a quasi-isomorphism for all $0 < j \leq N$ (respectively $-N \leq j < 0$) if and only if $Q_+$ (respectively $Q_-$) is acyclic.  Note that if some $\vhat_{\xi_0+j\cdot PD[\lambda]}$ vanishes on homology  for $0 < j \leq N$, then up to homotopy, $Q_+$ splits into a sum of two complexes, each with odd Euler characteristic, and thus $\dim H_*(Q_+) \geq 2$.  We have an analogous result for $\hhat_{\xi_0+j\cdot PD[\lambda]}$ and $Q_-$.   Note that these splittings exist because we are working with complexes over $\FF$.  

If $Q_+$ (respectively $Q_-$) is not acyclic, then by Proposition~\ref{prop:Vs=H-s}, $Q_-$ (respectively $Q_+$) is not acyclic.  Therefore, we see that if either $Q_+$ or $Q_-$ is not acyclic, then $H_*(Q)$ has dimension at least four.  This contradicts the fact that $H_*(\Xhat^N_{\xi_0}) = H_*(\Ahat_{\xi_0}) = \FF$, due to the exact triangle between $H_*(Q)$, $H_*(\Xhat^N_{\xi_0})$ and $H_*(\Ahat_{\xi_0})$.  Therefore, $Q$ is acyclic, and we see that the desired projection is a quasi-isomorphism.  
\end{proof}

The above argument shows that if $V_{\xi_0} > 0$, then $V_{\xi_0 + j \cdot PD[\lambda]} = H_{\xi_0 - j \cdot PD[\lambda]} = 0$ for all $j > 0$ when $Y_\lambda(K)$ is an L-space.  This will be useful for proving an analogue of Lemma~\ref{lem:proj-qi} for $\xi_0 + PD[\mu]$, which we now establish.

\begin{lemma}\label{lem:proj-qimu}
Suppose that $Y_\lambda(K)$ is an L-space and $V_{\xi_0} \geq 2$.  Then, the projection $\Piplus_{\xi_0 + PD[\mu]} : \Xplus_{\xi_0 + PD[\mu]} \to \Aplus_{\xi_0 + PD[\mu]}$ is a quasi-isomorphism.  
\end{lemma}

\begin{proof}
As discussed above, $V_{\xi_0 + j \cdot PD[\lambda]} = 0$ for all $j > 0$.  By \eqref{eq:Vs-inequality}, 
$$V_{\xi_0 + PD[\mu] + j \cdot PD[\lambda]} = 0$$
for all $j > 0$.  Therefore, the subcomplex consisting of the $\Ahat_{\xi}$ and $\Bhat_{\xi}$ with $\xi = \xi_0 + PD[\mu] + j \cdot PD[\lambda]$ with $j > 0$ is acyclic, so we quotient by this subcomplex.  Denote the result by $\Xhat'$, which has one-dimensional homology.    

By \eqref{eq:Vs-inequality} and the assumption that $V_{\xi_0}\geq 2$, we have that $V_{\xi_0 + PD[\mu]} \geq 1$, and thus $\vhat_{\xi_0 + PD[\mu]}$ is trivial on homology.  Choose $a \in \Ahat_{\xi_0 + PD[\mu]}$ and $b \in \Bhat_{\xi_0 + PD[\mu]}$ such that $a$ is a cycle generating the homology of $\Ahat_{\xi_0 + PD[\mu]}$ and $\partial b = \vhat_{\xi_0 + PD[\mu]}(a)$.  Then, $a + b \in {\Xhat'}$ is a cycle since $\hhat_{\xi_0 + PD[\mu]} \equiv 0$ in $\Xhat'$.  Of course, $a+b$ cannot be a boundary in $\Xhat'$, since $a$ is not a boundary, and we conclude that $a+b$ generates the homology of $\Xhat'$.  Since the projection onto $\Ahat_{\xi_0 + PD[\mu]}$ sends $a + b$ to $a$, we see that the projection from $\Xhat'$ to $\Ahat_{\xi_0 + PD[\mu]}$ is a quasi-isomorphism.  This is sufficient to yield the desired result.
\end{proof}

\begin{remark}\label{rmk:unknot-ok}
In the specific case that $K = K_U$ in $L(3,1)$, it can easily be computed from \eqref{eq:Vs-unknot} that $\Piplus_{\xi_0 + PD[\mu]}$ is a quasi-isomorphism, even though $V_{\xi_0} = 0$.  
\end{remark}

It remains to prove one more lemma before we are able to prove Proposition~\ref{prop:3k-1-surgery-formula}.  For notation, when $3k-1$ is even we write $[\theta] = \frac{3k-1}{2}[m] = \frac{1}{2}[\lambda]$.    

\begin{lemma}\label{lem:theta}
Suppose that $Y'$ is an L-space obtained from a distance one surgery on $L(3,1)$ with $|H_1(Y')| = 3k-1$ even.  Then, 
\begin{enumerate}
\item\label{theta:V} $V_{\xi_0 + PD[\theta]} = 0$,
\item\label{theta:qi} $\Pi^+_{\xi_0 + PD[\theta]}$ is a quasi-isomorphism,
\item\label{theta:spin} $G_{Y',K'}(\xi_0 + PD[\theta])$ is a self-conjugate $\spinc$ structure on $Y'$.  
\end{enumerate}    
\end{lemma}
\begin{proof}
\eqref{theta:V} Suppose that $V_{\xi_0 + PD[\theta]} > 0$.  By Proposition~\ref{prop:Vs=H-s}, we have that $H_{\xi_0 - PD[\theta]} > 0$.  Following the same arguments as in Lemma~\ref{lem:proj-qi} and \ref{lem:proj-qimu}, we see that $\Xhat^N_{\xi_0 + PD[\theta]}$ splits, up to homotopy, into a direct sum of three complexes with odd Euler characteristic.  This contradicts the fact that $Y'$ is an L-space.  

\eqref{theta:qi} We first deal with the case of $k \geq 3$.  It suffices to show that $H_{\xi_0 + PD[\theta]}$ and $V_{\xi_0 - PD[\theta]}$ are positive, because in that case, up to homotopy, $\Xhat_{\xi_0 + PD[\theta]}$ splits off a summand containing $\Ahat_{\xi_0 - PD[\theta]}, \Ahat_{\xi_0 + PD[\theta]}$ and $\Bhat_{\xi_0 + PD[\theta]}$ whose homology is necessarily rank one. Since $V_{\xi_0 + PD[\theta]} = H_{\xi_0 - PD[\theta]}$ by Proposition~\ref{prop:Vs=H-s}, the result follows.  By another application of Proposition~\ref{prop:Vs=H-s}, it suffices to simply establish the positivity of $H_{\xi_0 + PD[\theta]}$.  We will do this by showing that $H_{\xi_0 + PD[\theta]}$ is strictly greater than $V_{\xi_0 + PD[\theta]}$.  

Fix $n \gg 0$.  Let $\mft_*$ denote the $\spinc$ structure on $Y_{n\mu + \lambda}(K)$ such that $\Xi(\mft_*) = \xi_0 + PD[\theta]$.   Further, let $\mft_\pm$ denote $G_{Y,\pm K}(\xi_0 + PD[\theta])$ on $\spinc(Y)$.  By \cite[Theorem 7.1]{OSSmoothFour}, we have 
\begin{align*}
d(L(3,1), \mft_+)  - d(Y_{n\mu + \lambda}(K),\mft_*) &= \frac{c_1(\mfv_{\xi_0 + PD[\theta]})^2 - 3\sigma(W'_n) - 2\chi(W'_n)}{4} - 2V_{\xi_0 + PD[\theta]} \\
d(L(3,1), \mft_-)  - d(Y_{n\mu + \lambda}(K),\mft_*) &= \frac{c_1(\mfh_{\xi_0 + PD[\theta]})^2 - 3\sigma(W'_n) - 2\chi(W'_n)}{4} - 2H_{\xi_0 + PD[\theta]},
\end{align*}
since $\vplus_{\xi_0 + PD[\theta]}$ (respectively $\hplus_{\xi_0 + PD[\theta]}$) is given by the $\spinc$ cobordism map induced by $\mfv_{\xi_0 + PD[\theta]}$ (respectively $\mfh_{\xi_0 + PD[\theta]}$).  Consequently, it follows that $H_{\xi_0 + PD[\theta]} - V_{\xi_0 + PD[\theta]}$ is completely determined by homological information, so it suffices to show $H_{\xi_0 + PD[\theta]} > V_{\xi_0 + PD[\theta]}$ for our model knot $K_U$ in $L(3,1)$, described in Section~\ref{sec:mc-spinc}.
Recall from the proof of Lemma \ref{lem:spin-surg} that $V_\xi = \mathcal{V}_{\lfloor \frac{g(\xi)}{3} \rfloor}$ and $H_\xi = \mathcal{H}_{\lfloor \frac{g(\xi)}{3} \rfloor}$, and that $g(\xi_0)=1$. Using this and \eqref{eq:Vs-unknot},  since $k \geq 3$ we have that
\begin{align}
\label{theta:Hs=Vs+s} H_{\xi_0 + PD[\theta]} = \mathcal{H}_{\lfloor \frac{3k+1}{6} \rfloor} &= \mathcal{V}_{\lfloor \frac{3k+1}{6} \rfloor} + \left \lfloor \frac{3k+1}{6} \right \rfloor \\
\nonumber &> \mathcal{V}_{\lfloor \frac{3k+1}{6} \rfloor} \\
\nonumber &= V_{\xi_0 + PD[\theta]}.  
\end{align}
This completes the proof in the case that $k \geq 3$.  

It remains to deal with the case of $k = 1$.  This will not be needed in the application of Proposition~\ref{prop:3k-1-surgery-formula}, but we include it for completeness.  In this case, $[\lambda] = 2[m]$ and $[\theta] = [m]$.  While we do not have the strict inequality of  \eqref{theta:Hs=Vs+s}, a similar computation shows that 
$$V_{\xi_0 - PD[\theta] - PD[\lambda]} = H_{\xi_0 + PD[\theta] + PD[\lambda]} = V_{\xi_0 + PD[\theta] + PD[\lambda]} + 1 > 0.$$  
From this, it follows that $\Xhat_{\xi_0 + PD[\theta]}$ splits off, up to homotopy, the summand in Figure~\ref{fig:theta}.  
\begin{figure}
\[
\xymatrix{\Ahat_{\xi_0 - 3PD[m]} \ar[dr] & \Ahat_{\xi_0 - PD[m]} \ar[d] \ar[dr] & \Ahat_{\xi_0 + PD[m]} \ar[d] \ar[dr] & \Ahat_{\xi_0 + 3PD[m]} \ar[d] \\
& \Bhat_{\xi_0 - PD[m]} & \Bhat_{\xi_0 + PD[m]} & \Bhat_{\xi_0 + 3PD[m]}
}
\]
\caption{When $k = 1$, up to homotopy, $\Xhat_{\xi_0 + PD[\theta]}$ splits off the summand shown above.}\label{fig:theta}
\end{figure}
Using Proposition~\ref{prop:Vs=H-s}, we can apply similar arguments to Lemmas~\ref{lem:proj-qi} and ~\ref{lem:proj-qimu} to deduce that $\Pihat_{\xi_0 + PD[\theta]}$ is a quasi-isomorphism, which is sufficient.

\eqref{theta:spin} Recall that $G_{Y',K'}(\xi_0)$ is self-conjugate by Lemma~\ref{lem:spin-surg}.    Further, in $H^2(Y')$ we have that $i^*PD[\theta] = -i^*PD[\theta]$.  From Equation~\eqref{eq:G-equivariance} we obtain
\begin{align*}
G_{Y',K'}(\xi_0 + PD[\theta]) &= G_{Y',K'}(\xi_0) + i^*PD[\theta] \\
&= \overline{G_{Y',K'}(\xi_0)} - i^*PD[\theta] \\
&= \overline{G_{Y',K'}(\xi_0 + PD[\theta])}.
\end{align*}
\end{proof}

\subsection{The proofs of the surgery formulas}\label{sec:mc-proof}

\begin{proof}[Proof of Proposition~\ref{prop:3k-1-surgery-formula}]
We first establish Equation~\eqref{eq:3k-1-0}.  Suppose that $Y'$ is an L-space obtained from a distance one surgery on a knot in $L(3,1)$, where $|H_1(Y')| = 3k - 1 > 0$.  We would like to see that if $\mft = G_{Y',K'}(\xi_0)$, then,
\[
d(Y',\mft) = d(L(3k-1,3),1) - 2V_{\xi_0}.  
\] 
By Lemma~\ref{lem:spin-surg}, we know that $\mft$ is self-conjugate, so this will give the desired result.  

Since $\Piplus_{\xi_0}$ is a quasi-isomorphism (Lemma~\ref{lem:proj-qi}), the $d$-invariant of $Y'$ in the $\spinc$ structure $G_{Y',K'}(\xi_0)$ is computed by the minimal grading of a non-zero element of $H_*(\Aplus_{\xi_0})$, after the appropriate grading shift mentioned above Theorem~\ref{thm:mappingcone}.  Notice that before this grading shift, this minimal grading in $H_*(\Aplus_{\xi_0})$ is given by exactly $d(Y,G_{Y,K}(\xi_0)) - 2V_{\xi_0}$.  As described in \cite[Section 7.2]{OSRational}, the absolute grading shift on the mapping cone depends only on homological information, not on the isotopy type of the knot.  Let $\sigma(\xi)$ denote the grading shift applied to $\Aplus_\xi$ in the mapping cone formula, which does not depend on $K$.   In particular,
\begin{equation}\label{eq:grading-shift}
d(Y',G_{Y',K'}(\xi_0)) = d(Y,G_{Y,K}(\xi_0)) - 2V_{\xi_0} + \sigma(\xi_0).
\end{equation}

Consider the case of the knot $K_U$.  By the proof of Lemma~\ref{lem:spin-surg}, we have that $Y' = L(3k-1,3)$, $V_{\xi_0} = 0$, and $1$ corresponds with the $\spinc$ structure $G_{Y',K'}(\xi_0)$.  Consequently, 
\[
\sigma(\xi_0) = d(L(3k-1,3),1) - d(Y,G_{Y,K}(\xi_0)).
\]
For a knot $K \subset Y$ satisfying the hypotheses of the proposition, \eqref{eq:grading-shift} now implies    
\[
d(Y', \mft) = d(L(3k-1,3),1) -  2V_{\xi_0}.
\]
This establishes \eqref{eq:3k-1-0}.

The proof of Equation~\eqref{eq:3k-1-1} now follows the same strategy.  The only changes to the argument are that Lemma~\ref{lem:proj-qi} is replaced by Lemma~\ref{lem:proj-qimu} and Remark~\ref{rmk:unknot-ok}, and we must use that the $\spinc$ structure on $L(3k-1,3)$ given by $G_{Y',K'_U}(\xi_0 + PD[\mu])$ corresponds, up to $\spinc$-conjugation, with $4$.  To see this final claim, we use \cite[Section 6]{CochranHorn}\footnote{What is denoted as $[m]$ in this article is denoted $[\mu]$ in the notation of \cite{CochranHorn}.  Conveniently, the instances of $k$ used in each article agree in the case of $L(3k-1,3)$.}, where it is shown that 
the difference of the $\spinc$ structures corresponding to $i, j$ on $L(3k-1,3)$ is $\pm i^*((i-j)k \cdot PD[m]) \in H^2(L(3k-1,3))$.  (This is true even in the case that $3k - 1 = 2 < 3$.)  Since  $[\mu] = 3[m]$ and $G_{Y',K'_U}(\xi_0)$ is self-conjugate on $L(3k-1,3)$, we have the desired claim.  Since $G_{Y',K'_U}(\xi_0)$ corresponds to $i = 1$ on $L(3k-1,3)$, the claim follows.

Finally, we must establish that if $d(Y',\tilde{\mft}) \neq \frac{1}{4}$ for a self-conjugate $\spinc$ structure $\mft$, then $\mft = \tilde{\mft}$.  This only requires proof in the case that $3k-1$ is even.  Either $\tilde{\mft}=\mft$ or $\tilde{\mft} = G_{Y',K'}(\xi_0 + PD[\theta])$ by Lemmas~\ref{lem:spin-surg} and ~\ref{lem:theta}.  Applying the same argument as in the above cases, it follows from Lemma~\ref{lem:theta} that 
\[
d(Y', G_{Y',K'}(\xi_0 + PD[\theta])) = d(L(3k-1,3),\frac{3k+1}{2}) = \frac{1}{4}.
\]
Since $d(Y',\mft) \neq d(Y',G_{Y',K'}(\xi_0 + PD[\theta]))$ by assumption, we have that $\mft = \tilde{\mft}$.  
\end{proof}

\begin{proof}[Proof of Proposition~\ref{prop:3k+1-surgery-formula}]
The proof follows similarly to that of Proposition~\ref{prop:3k-1-surgery-formula}.  The main issue is that, as described in Lemma~\ref{lem:cobordism}, the two-handle attachment from $L(3,1)$ to $Y'$ is negative-definite instead of positive-definite.  Therefore, we must reverse orientation in order to obtain a positive-definite cobordism from $-L(3,1)$ to $-Y'$.  We now can repeat the arguments as before nearly verbatim, including Sections~\ref{sec:mc-preliminaries}-\ref{sec:mc-truncation}.  The only change is the ``model'' computation, which comes from the link $L^J$ as in Section~\ref{sec:mc-spinc}, where we use $-3$-surgery on $L_2$ and $+k$-surgery on $L_1$.  A slam-dunk shows that in the case of $J = U$, the result is $L(3k+1,3)$.  Repeating the arguments for Proposition~\ref{prop:3k-1-surgery-formula}, we obtain the terms coming from $L(3k+1,3)$ and $-Y'$.
\end{proof}

\section{Relevance of Theorem \ref{thm:main} and Corollary~\ref{cor} to DNA topology}
\label{sec:applications}
In subsection \ref{subsec:applications-band} we first give precise definitions of coherent and non-coherent band surgery and discuss implications of Theorem \ref{thm:main} and Corollary \ref{cor}. In subsection \ref{subsec:applications-DNA} we discuss the biological motivation for our specific focus on the trefoil and other $T(2, n)$ torus links.

\subsection{Modeling local reconnection by band surgery}\label{subsec:applications-band}
If $L$ is a link in the three-sphere, then a band $b:I\times I\rightarrow S^3$ is an embedding of the unit square such that $L\cap b(I\times I) = b(I \times \partial I)$. Two links $L_1$ and $L_2$ are related by a band surgery if 
$L_2 = ( L_1 - b(I\times \partial I) ) \cup  b(\partial I \times I)$. 
If $L_1$ and $L_2$ are oriented, and the orientation of $L_1 - b(I\times\partial I)$ is consistent with the orientations of both $L_1$ and $L_2$, then the band surgery is called \emph{coherent}. Otherwise, the band surgery is \emph{non-coherent}.\footnote{Note this definition does not imply the induced surface cobordism from $L_1$ to $L_2$ is orientable (respectively, non-orientable). For example, given a coherent band surgery from a two-component link to a knot, one may obtain a non-coherent band surgery via the same band move by reversing the orientation of one of the link components.} 
See Figure \ref{fig:resolutions}. 
Figure \ref{fig:bandings-1} illustrates non-coherent bandings transforming a knot to another knot. 
Note that a coherent band surgery necessarily changes the number of components of a link, as shown in Figure \ref{fig:bandings-2}. 

Write $(S^3, L_i) = (B, t_i) \cup (B', t')$, where $S^3 = B\cup B'$ is the union of two three-balls, with the sphere $\partial B = \partial B'$ intersecting $L_i$ transversely in four points, and where $t_i=(B \cap L_i)$ and  $t' = (B'\cap L_i)$. Here $(B, t_i)$ and $(B', t')$ are two-string tangles. It is often convenient to isotope $L_1$ and $L_2$ so that a coherent or non-coherent band surgery can be expressed as the replacement of a rational $(0)$ tangle by an $(\infty)$ or $(\pm 1/n)$ tangle (Figure \ref{fig:moreresolution}). When $n$ is small, these tangles have special relevance in biology (see for example \cite{Sumners1995, VazquezSumners, Vazquez2005, Shimokawa, Stolz2017}). For example, in the context of DNA recombination, the local reconnection sites correspond to the core regions of the recombination sites, i.e. two very short DNA segments where cleavage and strand-exchange take place. Thus these tangle replacements and the corresponding band surgeries appropriately model the recombination reaction.  Note that $(B, t_1)$ is replaced with $(B, t_2)$ leaving $(B', t')$ fixed. In terms of the resulting tangle decomposition, this simplification comes at the expense of complicating the outside tangle $(B', t')$.   
 
\begin{figure}
\includegraphics[height = 0.7in]{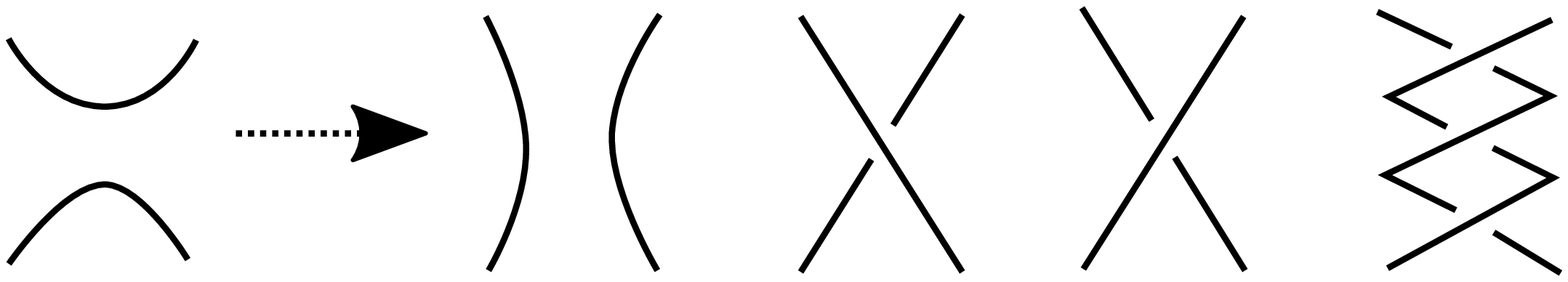}
\caption{Examples of rational tangle replacements. We model local reconnection as a tangle replacements, such as those pictured. We typically assume low-crossing tangle replacements. Any additional topological complexity surrounding the reconnection sites is pushed to the outside tangle, which remains fixed during reconnection. }
\label{fig:moreresolution}
\end{figure}

The double cover of $B'$ branched over $t'$ is a compact, connected, oriented 3-manifold $M$ with torus boundary. The manifold $M$ may also be obtained as $\Sigma(L_1)-\mathcal{N}(K)$, where we write $\Sigma(L)$ to denote the double cover of $S^3$ branched over $L$, and the knot $K$ is the lift of a properly embedded arc arising as the core of the band; the latter perspective has been adopted throughout the current article. Both $\Sigma(L_1)$ and $\Sigma(L_2)$ are obtained by Dehn fillings of $M$, and the Montesinos trick \cite{Montesinos} implies that these fillings are distance one. One such example is illustrated in Figure \ref{fig:tangle}. For this reason, Theorem \ref{thm:main} immediately provides an obstruction to the existence of band surgeries between the right-handed trefoil knot and the torus link $T(2,n)$ for $n\neq \pm 1, \pm 2, 3, 4, -6$, and 7. In section \ref{subsec:applications-DNA}, we present examples from the literature where most of the exceptional cases have been observed in DNA recombination reactions involving the trefoil. 
\begin{figure}
\begin{center}
\begin{tikzpicture}

\node[anchor=south west,inner sep=0] at (0,0) {\includegraphics[width=4.2in]{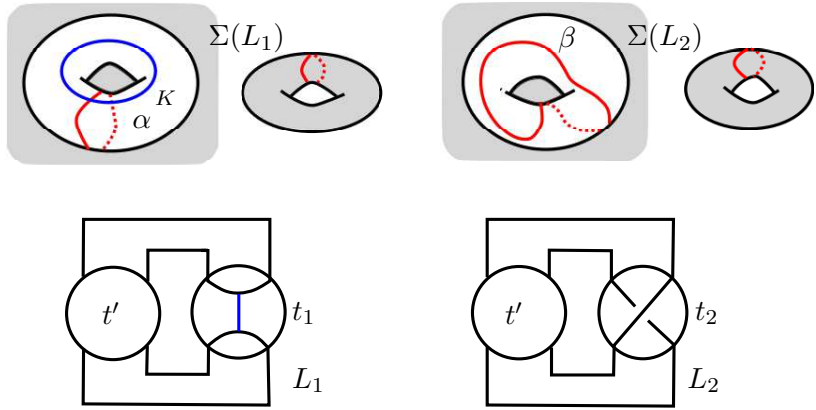}};


\node[label=above right:{$L_1$}] at (3.5,0.0){};
\node[label=above right:{$L_2$}] at (8.7,0.0){};
\node[label=above right:{\scriptsize{$K$}}] at (1.7,3.8){};'''
\node[label=above right:{$\alpha$}] at (1.4,3.5){};
\node[label=above right:{$\beta$}] at (7.0,4.5){};
\node[label=above right:{$\Sigma(L_1)$}] at (2.4,4.5){};
\node[label=above right:{$\Sigma(L_2)$}] at (7.9,4.5){};
\node[label=above right:{$t'$}] at (1.0,0.9){};
\node[label=above right:{$t'$}] at (6.3,0.9){};
\node[label=above right:{$t_1$}] at (3.5,0.9){};
\node[label=above right:{$t_2$}] at (8.8,0.9){};

\end{tikzpicture}
\end{center}
\caption{An example of a rational tangle replacement realizing a band surgery, together with the corresponding lift to the branched double cover. The blue arc properly embedded in the complement of $L_1$ lifts to the blue knot $K$ in $\Sigma(L_1)$. The exterior tangle $(B', t')$ is of arbitrary complexity and lifts to $M$, depicted as the open shaded area. Dehn fillings along the curves $\alpha$ and $\beta$ yield $M(\alpha) = \Sigma(L_1)$ and $M(\beta)=\Sigma(L_2)$, respectively. When drawn on the same boundary torus, $\alpha$ and $\beta$ intersect geometrically once.}
\label{fig:tangle}
\end{figure}

Coherent band surgery is better understood than the non-coherent case (see for example \cite{IshiharaShimokawa, ISV, BuckIshihara2015, BIRS, DIMS, Shimokawa}). If a coherent band surgery decreases the maximal Euler characteristic of an oriented surface without closed components bounding the link, it is known that the band can be isotoped to lie onto a taut Seifert surface \cite{ST}, \cite[Theorem 1.6]{HS}. Thus minimal Seifert surfaces can sometimes be used to obstruct the existence of coherent band surgeries or characterize the tangle decompositions that yield existing surgeries \cite{DIMS, BuckIshihara2015, BIRS}. As non-coherent band surgery is an unoriented operation, such techniques are not immediately available. There are several obstructions to the existence of a non-coherent band surgery coming from certain evaluations of the Jones or Q-polynomials \cite{AbeKanenobu}, 
but these are not helpful in the present case. A theorem of Kanenobu \cite[Theorem 2.2]{Kanenobu} implies that if a knot or link $L$ is obtained from an unknotting number one knot $K$ by a coherent or non-coherent band surgery, then either $2\det(L)$ or $-2\det(L)$ is a quadratic residue of $\det(K)$. Because this condition is always true when $\det(K)=3$, the obstruction is inapplicable in the case of the trefoil. Theorem \ref{thm:main} provides new obstructions to the existence of both coherent and non-coherent band surgeries along the trefoil.

\subsection{Relevance to DNA topology}\label{subsec:applications-DNA}

DNA is a nucleic acid that carries the genetic code of an organism. In its most common form, the B-form, DNA is a right-handed double helix with two sugar-phosphate backbones lined up by nitrogeneous bases A, T, C and G. The sequence of bases determines the genetic code. The bases along one backbone are complementary to the bases along the second backbone, and are held together via hydrogen bonds between A and T and between C and G. The length of a DNA molecule is measured in the number of nucleotides, or base-pairs (bp). For example, the genomes of viruses such as bacteriophages can be fairly short, while the circular chromosome of {\it Escherichia coli} ranges from 4.5 to 5.5 million bp, and the human genome is approximately 3 billion bp long.

\paragraph{\it Importance of $T(2, n)$ torus knots and links in recombination.} In the early 1960s, the Frisch-Wasserman-Delbr\"{u}ck Conjecture \cite{FW, DM} stated that in long polymer chains knots would occur with almost sure certainty. The conjecture has been proved for various polymer models  \cite{SW,D1,DPS}. It was also verified experimentally on randomly circularized DNA chains \cite{Liu81, Rybenkov93, Shaw93, Arsuaga02}. The high knotting probability is accentuated when the polymer chains occur in confined volumes, such as a long chromosome inside a viral capsid or in a cell nucleus. In studies dealing with geometry and topology of long DNA molecules, double-stranded DNA is modeled as the curve drawn by the axis of the double-helix. Experimental and numerical work of closed polymer chains in open space clearly indicate that {\it the most probable knot is the trefoil knot}. 

In addition to the trefoil knot, other $T(2,n)$ torus knots and links are especially relevant in biology as illustrated by the effects of replication on circular DNA. In our current understanding, the tree of life consists of three domains: Archaea, Bacteria and Eukarya. Bacteria and Archaea have circular chromosomes. The process of DNA replication on a circular chromosome, whereby the cell produces a copy of its genome in preparation for cell division, yields two interlinked daughter DNA molecules. The two-component links defined by the axes of the DNA double-helices have been shown experimentally to be $T(2, n)$ torus links \cite{AdamsCozz92}. Note that this is a consequence of the right-handed double-helical structure of DNA. The two components must be unlinked to ensure survival of the next generation of cells. Typically the unlinking is mediated by type II topoisomerases, enzymes that introduce a double-stranded break and mediate strand-passage. The local action of type II topoisomerases can be modeled as a crossing change. However, Grainge {\it et al.} \cite{Grainge2007} showed that unlinking of replication links can also be mediated by recombination and proposed an unlinking mechanism by local reconnection where each $T(2,n)$ torus link was converted to a $T(2, n-1)$ torus knot, and each $T(2, n-1)$ torus knot was converted to a $T(2,n-2)$ torus link. In \cite{Shimokawa} it was proved that this mechanism of stepwise unlinking is the only possible pathway that strictly reduces the complexity (measured as the minimal crossing number) of the DNA substrates at each step. More recently, using a combination of analytical and numerical tools, \cite{Stolz2017} showed that even when no restrictions are imposed on the reduction in crossing number, the stepwise mechanism proposed in \cite{Grainge2007} is the most likely. These examples underscore the importance of understanding any topological transitions between torus knots and links, including the trefoil.
 
\paragraph{\it Band surgery as a model for DNA recombination.} As was previously mentioned, the local action of recombination enzymes can be thought of as a simple reconnection and can be modeled mathematically as band surgery. The reconnection sites are two short, identical DNA segments (typically 5-50bp long). They usually consist of a non-palindromic sequence of nucleotides and we can therefore assign an unambiguous orientation to each site. Two reconnection sites in a single circular chain may induce the same orientation along the chain, in which case they are said to be in {\it direct repeats}. If the sites induce opposite orientations into the chain, they are said to be in {\it inverted repeats} (Figure \ref{fig:SiteOrientation}). When the substrate is a knot with two directly repeated sites, reconnection yields a product with two components, which may be non-trivially linked. This process corresponds to a coherent band surgery. Conversely, if the substrate is a two-component link with one site on each component, the product is a knot with two directly repeated sites. When the substrate is a knot with two inversely repeated sites, the product is a knot with the same site orientation. This corresponds to a non-coherent band surgery. 
 
Two-string tangle decompositions are commonly used to model enzymatic complexes attached to two segments along a circular DNA molecule. The topology of the product depends on the global conformation adopted by the substrate prior to reconnection. Therefore understanding the outside tangle $(B',t')$ is crucial to an accurate description of the enzymatic reaction. When the tangles involved are rational or sums of two rational tangles, there is a well-known combinatorial technique, called the \emph{tangle calculus}, which allows one to solve systems of tangle equations related to an enzymatic action and thus infer mechanisms of the enzymes. The tangle method was first proposed by Ernst and Sumners in \cite{ES-1} and is now standard in the toolkit of DNA topologists. Coherent and non-coherent band surgeries fit easily into this framework.

{\it Site-specific recombination experiments consistent with Corollary \ref{cor}.} 
DNA recombination events occur often in the cellular environment since they are needed for repair of double-stranded breaks, a deleterious form of DNA damage. Most of the time double-strand breaks are properly repaired by a process called homologous recombination, and no visible changes are present on the DNA at the end of the process. However, sometime homologous recombination results in local reconnection, also called {\it cross-over} in the biological literature. Studying topological changes related to homologous recombination is difficult due to the length of the DNA substrates. Site-specific recombination is another recombination process that has been extensively studied from the topological point of view. Site-specific recombinases are important to a variety of naturally occurring processes and genetic engineering techniques, such as the integration or excision of genetic material \cite{gottesman1971}, dimer resolution \cite{Stirling1988}, or the regulation of gene expression via inversion \cite{heichmanjohnson1990}. The trefoil knot and other $T(2,n)$ torus knots and links have commonly been used as substrates, or observed as products of site-specific recombination. In the next few paragraphs, we survey a few examples from the literature specific to Corollary \ref{cor}. Recall that site-specific recombinases target short DNA sequences called recombination sites. By convention the names of the sites are short words indicated in italics (e.g. \emph{att, dif, psi, res}). The names of the enzymes are capitalized (e.g. Xer, $\lambda$-Int, Gin, Hin). 
 
In Spengler {\it et al.} \cite{Spengler85} the authors incubated a 9.4 kilobase (kb) negatively supercoiled DNA plasmid containing two inversely repeated {\it att} recombination sites, with the integrase $\lambda$-Int from bacteriophage $\lambda$. The products were knots with odd number of crossings, and their gel migration was consistent with that of torus knots. The analogous experiment with plasmids carrying two directly repeated {\it att} sites yielded two-component links with even number of crossings. Crisona {\it et al.} \cite{CrisonaCozz99} confirmed that all products of $\lambda$-Int recombination on unknotted substrates with two recombination sites are right-handed torus knots (in the inverted repeat case) or torus links (in the direct repeat case) of the form $T(2,n)$.

There are many instances of coherent bandings in the biological literature. For example, in \cite{BenjaminCozz1990}, a 7kb substrate with two {\it att} sites in direct repeat and two {\it res} sites in direct repeat incubated with $\lambda$-Int produced right-handed torus links with antiparallel {\it res} sites. Links with 4, 6, 8, 10 crossings were observed. These links were then incubated with another enzyme, the Tn3 resolvase. In this study the trefoil knot clearly appeared as a product of resolvase recombination on a right-handed four and six-crossing torus link with two sites in anti-parallel orientation\footnote{These are the links $4^{2'}_1$ and $6^{2'}_1$, respectively, using the nomenclature convention from \cite{Stolz2017}.}. The trefoil obtained is predicted to be a negative trefoil, which is left-handed. This transition is the mirror to the transition between $T(2,3)$ and $T(2,-6)$ from Corollary \ref{cor}.

The Xer site-specific recombination system is a good source of examples relevant to the results from Corollary \ref{cor}. In the cell, Xer recombination is known to act at two directly repeated {\it dif} sites along the bacterial chromosome to resolve chromosomal dimers, and has been shown to unlink replication links \cite{Grainge2007}. The enzymatic action is consistent with a stepwise unlinking pathway \cite{Grainge2007, Shimokawa}. In this pathway any $T(2,4)$ link with parallel sites is converted to a $T(2,3)$ knot and any $T(2,3)$ knot is converted to a $T(2,2)$ link. In a different reaction, Xer recombination at two {\it psi} sites in direct repeats converts an unknot to a $T(2,4)$ link with anti-parallel sites (see Figure \ref{fig:SiteOrientation})\cite{Vazquez2005}. The {\it psi} sites are 28bp long and consist of an 11bp XerC binding region, an 11bp XerD binding region, a 6bp asymmetric central region, and a 160bp accessory sequence adjacent to the XerC binding site. In Bregu {\it et al.} \cite{BreguSherratt2002}, the 28bp core region of the {\it psi} site was inverted with respect to the accessory sequence.  This allowed the authors to mediate Xer recombination on sites in inverted repeats, i.e. the non-coherent case. The reaction converted an unknot to a trefoil $T(2,3)$.

\begin{figure}
\includegraphics[height = 1.4in]{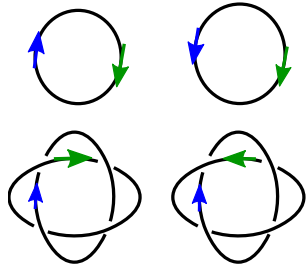}
\caption{(Upper) Relative orientations of the reconnection sites. The figure shows an unknotted chain with two sites in direct repeats (left) and an unknot with two sites in inverted repeats (right). (Lower) $T(2,n)$ torus links with parallel orientation of the strands and linking number $+n/2$ (left) and antiparallel orientation of the strands and linking number $-n/2$ (right). }
\label{fig:SiteOrientation}
\end{figure}

Non-coherent bandings have been observed experimentally in the action of several other site-specific recombinases. Noteworthy are enzymes Gin and Hin. Gin is a site-specific recombinase from bacteriophage Mu used to change the genetic code of the viral genome by inverting one of the DNA arcs, called the G-segment, bound by the recombination sites \cite{VazquezSumners}. Gin acts processively, i.e. it performs several rounds of recombination before releasing its substrate. In the first round, when acting on an unknotted DNA circle with sites in inverted repeat, Gin produces an unknot with an inverted G-segment, and in the second round the unknot is turned into a trefoil, and the original genetic sequence is restored. By a similar mechanism, Hin converts its unknotted substrate to a trefoil \cite{BuckMauricio}. Hin does not change the genetic code of the DNA. These examples illustrate the transition between $T(2,3)$ and $T(2,1)$. 

In sum, we have presented examples from the literature where some of the exceptional cases from Theorem \ref{thm:main} have been observed. In particular transitions between the right-handed trefoil knot and the torus links $T(2,n)$ for $n= \pm 1, \pm 2, 4$, and between the left-handed trefoil and the torus links $T(2,-4)$ and $T(2,6)$ have been reported. We note that in the non-coherent case, the transitions observed were from the unknot to the trefoil. Transitions from the trefoil to the trefoil, and from the trefoil to $T(2,7)$ are probably very rare. The frequency of such transitions can be assessed using numerical simulations as described in \cite{Stolz2017}. In fact, in a preliminary numerical experiment where non-coherent band surgery is modeled on $9.6\times10^4$ trefoils represented as polygonal chains in the simple cubic lattice, the probability of the transition from the trefoil to the unknot was $0.975$, from the trefoil to itself was $0.013$, and the transition from the trefoil to $T(2,7)$ was not observed. In a separate experiment where $3.3\times10^5$ polygons of type $T(2,7)$ with two sites in inverted repeats were used as substrates, the transition to the unknot occurred with probability $0.94$ and to the trefoil with probability $0.008$. In this experiment one single transition was observed from $T(2,7)$ to itself.

\subsection*{Acknowledgements.}  We would like to thank Ken Baker, Yi Ni, Peter Ozsv\'ath, Koya Shimokawa, Laura Starkston, and Faramarz Vafaee for helpful conversations. We thank Michelle Flanner for assistance obtaining preliminary numerical simulations. TL was partially supported by DMS-1709702.  AM and MV were partially supported by DMS-1716987.  MV was also partially supported by CAREER Grant DMS-1519375.

\pagebreak
\begin{footnotesize}
	\bibliographystyle{plain}
	\bibliography{biblio}
\end{footnotesize}

\end{document}